\documentclass[12pt]{article}

\usepackage{latexsym}
\usepackage{amssymb}
\usepackage{amsmath}
\usepackage{amsfonts}
\usepackage{textcomp}
\usepackage{amscd}
\usepackage{amsthm}
\usepackage{amsxtra}
\usepackage{stmaryrd}
\usepackage{graphicx}
\usepackage[all]{xy}

\numberwithin{equation}{section}
\newtheorem{theorem}[equation]{Theorem}

\newtheorem{corollary}[equation]{Corollary}
\newtheorem{lemma}[equation]{Lemma}
\newtheorem{proposition}[equation]{Proposition}

\theoremstyle{definition}

\newtheorem{remark}[equation]{Remark}

\def\limind{\mathop{\oalign{{\rm lim}\cr
\hidewidth$\longrightarrow$\hidewidth\cr}}}
\def\limproj{\mathop{\oalign{{\rm lim}\cr
\hidewidth$\longleftarrow$\hidewidth\cr}}}

\def\Z{{\mathbb {Z}}_p}
\def\Q{{\mathbb {Q}}_p}

\def\G{{\rm GL }_2(\Q)}

\def\ra{\rightarrow}
\def\ZZ{{\mathbb Z}}

\def\Qp{{\mathbb{Q}}_p}
\def\Cp{{\mathbb{C}}_p}
\def\Zp{{\mathbb{Z}}_p}
\def\eps{\varepsilon}
\def\OO{\mathcal{O}}

\def\aa{\textbf{A}}
\def\bb{\textbf{B}}
\def\ee{{\textbf{E}}}

\def\at{\widetilde{\textbf{A}}}
\def\bt{\widetilde{\textbf{B}}}
\def\et{\widetilde{\textbf{E}}}
\def\aplus{\textbf{A}^+}
\def\bplus{\textbf{B}^+}

\def\atplus{\widetilde{\textbf{A}}^+}
\def\btplus{\widetilde{\textbf B}^+}
\def\etplus{\widetilde{\textbf{E}}^+}

\def\bbQ{\mathbb{Q}}
\def\calE{\mathcal{E}}
\def\calR{\mathcal{R}}
\def\Gal{\mathrm{Gal}}
\def\rig{\mathrm{rig}}
\def\dif{\mathrm{dif}}
\def\res{\mathrm{res}}
\def\Res{\mathrm{Res}}
\def\D{\textbf{D}}

\DeclareMathOperator{\rank}{rank}
\DeclareMathOperator{\Ind}{Ind}
\def\L{\mathbb{Q}_p(\varepsilon^{(n)})}

\def\m{(\varphi,\Gamma)}

\def\r{\mathcal{R}}
\newcommand{\btdag}[1]{\widetilde{\textbf{B}}^{\dagger #1}}
\newcommand{\atdag}[1]{\widetilde{\textbf{A}}^{\dagger #1}}
\newcommand{\bdag}[1]{\textbf{B}^{\dagger #1}}
\newcommand{\adag}[1]{\textbf{A}^{\dagger #1}}

\newcommand{\btrig}[2]{\textbf{B}^{\dagger #1}_{\mathrm{rig} #2}}

\begin{document}

\title{Cohomology and Duality for $(\varphi, \Gamma)$-modules over the Robba Ring}

\author{Ruochuan Liu}

\author{Ruochuan Liu\\ Department of Mathematics, Room 2-229\\ Massachusetts Institute of Technology\\
77 Massachusetts Avenue\\ Cambridge, MA 02139\\ ruochuan@math.mit.edu}
\maketitle

\begin{abstract}
Given a $p$-adic representation of the Galois group of a local field,
we show that its Galois cohomology can be computed using the associated
\'etale $(\varphi, \Gamma)$-module over the Robba ring; this is a variant
of a result of Herr. We then establish analogues, for not necessarily \'etale
$(\varphi, \Gamma)$-modules over the Robba ring, of
the Euler-Poincar\'e characteristic formula and Tate local duality for
$p$-adic representations. These results are expected to intervene in the duality theory for Selmer
groups associated to de Rham representations.
\end{abstract}

\tableofcontents

\section*{Introduction}

Two of the basic results in the theory of Galois cohomology over a local field
are the Euler-Poincar\'e characteristic formula and Tate's local duality
theorem. In this paper, we generalize these results to a larger category
than the category of $p$-adic representations, namely the category of
$(\varphi, \Gamma)$-modules over the Robba ring. We expect that these results
will be relevant to the deformation theory of Galois representations, via a
study of duality properties of Selmer groups associated to
de Rham representations. This would extend the work of
Bloch-Kato for ordinary representations \cite{BK90}. See \cite{P} for more details.

In the remainder of this introduction, we formulate more
precise statements of our results (but we skip some definitions
found in the body of the text).
Let $p$ be a prime number,
and fix a finite extension $K$ of $\bbQ_p$.
Write $G_K = \Gal(\overline{K}/K)$ and $\Gamma =
\Gamma_K = \Gal(K(\mu_{p^\infty})/K)$.
By a ``$p$-adic representation'' let us mean a finite dimensional
$\bbQ_p$-vector space $V$ equipped with a continuous linear action
of $G_K = \Gal(\overline{K}/K)$. Fontaine \cite{Fo91} constructed
a functor $\D$ associating to each $p$-adic representation $V$ an
\emph{\'etale $(\varphi, \Gamma)$-module} $\D(V)$ over a certain two-dimensional
local field $\calE_K$, and established an equivalence of categories between
$p$-adic representations and \'etale $(\varphi, \Gamma)$-modules.
For the moment, all we will say about $\D(V)$ is that
$\calE_K$ is constructed to carry a Frobenius operator
$\varphi$ and an action of $\Gamma$ commuting with each other,
and $\D(V)$ is a finite dimensional $\calE_K$-vector space carrying semilinear actions
of $\varphi$ and $\Gamma$. (The condition of \'etaleness is a certain extra
restriction on the $\varphi$-action.)

Fontaine's equivalence suggests that one can compute Galois cohomology of
a $p$-adic representation from the ostensibly simpler object $\D(V)$; this
was worked out by Herr \cite{H98}, who constructed an explicit
complex from $\D(V)$ computing the Galois cohomology of $V$ and the cup
product. In case
$\Gamma$ is procyclic and topologically generated by $\gamma$,
the complex is particularly easy to describe: it is the complex
\[
0 \to\D(V) \stackrel{d_1}{\to}\D(V) \oplus\D(V) \stackrel{d_2}{\to}\D(V) \to 0
\]
with $d_1(x) = ((\gamma - 1)x, (\varphi - 1)x)$ and $d_2((x,y)) =
(\varphi - 1)x - (\gamma - 1)y$. Herr also showed \cite{H01} that
one can easily recover the Euler-Poincar\'e characteristic formula
and Tate local duality from this description.

More recently, Berger \cite{LB02} (building on work of Cherbonnier
and Colmez \cite{CC98}) constructed a functor $\D^\dagger_{\rig}$
giving an equivalence between the category of $p$-adic
representations and the category of \'etale $(\varphi,
\Gamma)$-modules over a different ring, the \emph{Robba ring}
$\calR_K$. Berger's original justification for introducing
$\D^\dagger_{\rig}$ was to show that for $V$ a de Rham
representation, $\D^\dagger_{\rig}(V)$ can be used to construct a
$p$-adic differential equation of the sort addressed by Crew's
conjecture; this led Berger to prove that the $p$-adic monodromy
conjecture of $p$-adic differential equations implies Fontaine's
conjecture that de Rham representations are potentially semistable.

Subsequently, Colmez \cite{C05} observed that non-\'etale $(\varphi,
\Gamma)$-modules over $\calR_K$ play a role in the study of Galois
representations, even though they do not themselves correspond to
representations. Colmez specifically considered the class of
two-dimensional representations which are \emph{trianguline} (this
idea goes back to Mazur), that is, their associated \'etale
$(\varphi, \Gamma)$-modules over $\calR_K$ admit filtrations by not
necessarily \'etale $(\varphi, \Gamma)$-submodules with successive
quotients free of rank 1. (The supply of such representations is
plentiful: for instance, a result of Kisin \cite{Ki03} implies that
many of the Galois representations associated to overconvergent
$p$-adic modular forms are trianguline.) Colmez classified these
representations in the dimension 2 case and showed that they fit
naturally into the $p$-adic local Langlands correspondence of $\G$
initiated by Breuil.

In so doing, Colmez introduced an analogue of Herr's complex for an
arbitrary $(\varphi, \Gamma)$-module over $\calR_K$. Although he
does not explicitly assert that this complex computes Galois
cohomology, we infer that he had the following theorem in mind; we
include its proof, an easy reduction to Herr's theorem, to fill a
gap in the literature.

\begin{theorem}
Let $V$ be a $p$-adic representation of $G_K$. Then there are isomorphisms
\[
H^i(\D^\dagger_{\rig}(V)) \cong H^i(G_K, V) \qquad (i=0,1,2)
\]
which are functorial in $V$ and compatible with cup products.
\end{theorem}

At this point, one may reasonably expect that the Euler-Poincar\'e
characteristic formula and Tate local duality should extend to
$(\varphi, \Gamma)$-modules over the Robba ring, using
$\r(\omega) = \D^\dagger_{\rig}(\Qp(1))$ as the dualizing object. The main goal of
this article is to prove these results.
\begin{theorem}
Let $D$ be a $(\varphi, \Gamma)$-module over the Robba ring
$\calR_K$.
\begin{enumerate}
\item[(a)] We have $H^i(D)$ are all finite dimensional $\Qp$-vector spaces and
\[
\chi(D) = \sum_{i=0}^2 (-1)^i \dim_{\Qp} H^i(D) = -[K:\Qp]\rank D.
\]
\item[(b)] For $i=0,1,2$, the composition
\[
H^i(D) \times H^{2-i}(D^\vee(\omega)) \to
H^2((D \otimes D^\vee)(\omega)) \to
H^2(\omega),
\]
in which the first map is the cup product, is a perfect pairing
into $H^2(\omega) \cong \Qp$.
\end{enumerate}
\end{theorem}
Our method of proof is to reduce to the known case of an \'etale
$(\varphi, \Gamma)$-module, where by Theorem 0.1 we can invoke the
standard form of the theorem. In doing so, we construct a bigger
category, the category of generalized $\m$-modules, which allows us
to consider the cohomology of the quotient of two $\m$-modules.
Moreover, in case $K=\Q$ and $p>2$, we provide an explicit
calculation of $H^2$ of rank $1$ $\m$-modules as a complement to
Colmez's calculation on $H^0$ and $H^1$ in \cite{C05}.

The author should mention that in a different direction, Seunghwan
Chang has obtained some interesting results concerning extensions of
$\m$-modules in his thesis \cite{Ch07}.

\section{Preliminaries}
\subsection{$p$-adic Hodge theory and $\m$-modules} This section is a brief
summary of some basic constructions of $p$-adic
Hodge theory and $\m$-modules. The results recalled here can be found in
\cite{FW79}, \cite{F82}, \cite{Wi83}, \cite{Fo91}, \cite{CC98}, and \cite{LB02}.

Let $p$ be a prime number,
and fix a finite extension $K$ of $\bbQ_p$.
Write $G_K = \Gal(\overline{K}/K)$.
A \emph{$p$-adic representation} $V$ is a finite dimensional $\Qp$-vector space with a
continuous linear action of $G_{K}$. The dimension of this representation is defined as the dimension of $V$ as a $\Qp$-vector space and is usually denoted by $d$. \

Let $k$ be the residue field of $K$, $W(k)$ be the ring of Witt vectors with
coefficients in $k$, and $K_{0}=W(k)[1/p]$ be the maximal unramified subfield
of $K$. Let $\mu_{p^{n}}$ denote the group of $p^{n}$-th roots of unity. For
every $n$, we choose a generator $\eps^{(n)}$ of $\mu_{p^{n}}$, with the
requirement that  $(\eps^{(n+1)})^{p}=\eps^{(n)}$. That makes
$\eps=\limproj_{n}\eps^{(n)}$ a generator of
$\limproj_{n}\mu_{p^{n}}\simeq \Zp(1)$. We set $K_{n}=K(\mu_{p^{n}})$ and
$K_{\infty}=\bigcup_{n=1}^{\infty}K_{n}$. The cyclotomic character $\chi:
G_{K}\rightarrow \Zp^{\times}$ is defined by $g(\eps^{(n)})=(\eps^{(n)})^{\chi(g)}$
for all $n\in\mathbb{N}$ and $g\in G_{K}$. The kernel of $\chi$ is $H_{K}=\Gal(\overline{\Qp}/K_{\infty})$, and $\chi$ identifies
$\Gamma=\Gamma_{K}=G_{K}/H_{K}$ with an open subgroup of $\Zp^{\times}$.

Let $\textbf{C}_{p}$ denote the $p$-adic complex numbers, i.e. the completion of
$\overline{\Qp}$ for the $p$-adic topology, and set
\begin{center}
 $\et=\displaystyle{\limproj_{x\mapsto x^{p}}}\Cp =\{(x^{(0)},x^{(1)},...)\ |\
(x^{(i+1)})^{p}=x^{(i)}\}$.
\end{center}
A ring structure on $\et$ is given by the
following formulas:
If $x=(x^{(i)})$ and $y=(y^{(i)})$, then their sum $x+y$ and product $xy$
are given by
\begin{center}
 $(x+y)^{(i)}=\displaystyle{\lim_{j\rightarrow \infty }}(x^{(i+j)}+y^{(i+j)})^{p^{j}}$ and
$(xy)^{(i)}=x^{(i)}y^{(i)}$.
\end{center}
If $x=(x^{(i)})$, we define $v_{\et}(x)=v_{p}(x^{(0)})$.
This is a valuation on $\et$ and the corresponding topology coincides with the
projective limit topology; as a consequence, $\et$ is a complete valuation ring
with respect to $v_{\et}$. Furthermore, the induced $G_{\Qp}$-action on $\et$
preserves this valuation. Let $\etplus$ be the ring of integers of this
valuation; i.e. $\etplus$ is the set of $x\in \et$ such that $x^{(0)}\in
\OO_{\Cp}$. From the construction of $\eps$, we can naturally view it as an element of $\et^+$.
Set $\ee_{K_0}=k((\eps-1))$, $\ee$ the separable closure of $\ee_{K_0}$ in $\et$ and $\ee_K=\ee^{H_K}$.
If $K_0'$ denotes the maximal unramified extension of $K_0$ in $K_\infty$ and $k'$ is its residue field,
then the discrete valuation ring $\ee^+_K$ is just $k'[[\overline{\pi}_K]]$ where $\overline{\pi}_K$ is a uniformizer (\cite{FW79}, \cite{Wi83}).

Let $\atplus$ (resp. $\at$) be the ring $W(\etplus)$ (resp. $W(\et)$) of Witt vectors with coefficients in $\etplus$ (resp. $\et$), and
$\btplus=\atplus[1/p]$ (resp. $\bt=\at[1/p]$). Set
$\pi=[\eps]-1$, and $q=\varphi(\pi)/\pi$. Since $\etplus$ is perfect, we have
\begin{center}
 $\atplus=\{x=\sum_{k=0}^{\infty}p^{k}[x_{k}]\ | \ x_{k}\in \etplus\}$,
\end{center}
where $[x_{k}]$ is the Teichm\"{u}ller lift of $x_{k}$ in $\atplus$.
This gives a bijection $\atplus\rightarrow(\etplus)^{\mathbb{N}}$
which sends $x$ to $(x_{0},x_{1},...)$. Let $\atplus$ be endowed
with the topology induced from the product topology of the right
hand side. Another way to get this topology is to define
$([\overline{\pi}]^{k},p^{n})$ as a basis of neighborhoods of $0$.
The topology of $\at$ is defined in the same way. The absolute
Frobenius $\varphi$ of $\et$ lifts by functoriality of Witt vectors
to the Frobenius operator $\varphi$ of $\at$ which commutes with the
Galois action. It is easy to see that
\begin{center}
$\varphi(\sum_{k=0}^{\infty}p^{k}[x_{k}])=\sum_{k=0}^{\infty}p^{k}[x_{k}^{p}]$
\end{center}
and therefore $\varphi$ is an isomorphism. Now let $\aa_{K_0}$ be
the completion of $\OO_{K_0}[\pi,\pi^{-1}]$ in $\at$ for the
topology given above. It is also the completion of
$\OO_{K_0}[[\pi]][\pi^{-1}]$ for the $p$-adic topology. This is a
Cohen ring with residue field $\ee_{K_0}$. Let $\bb$ be the
completion for the $p$-adic topology of the maximal unramified
extension of $\bb_{K_0}=\aa_{K_0}[1/p]$. We then define
$\aa=\at\cap\bb$ and $\aplus=\atplus\cap\bb$. Note that these rings
are endowed with the induced $G_{\Q}$ and Frobenius actions from
$\bt$. For $S$ any one of these rings, define $S_{K}=S^{H_{K}}$.
Therefore $\bb_K=\aa_K[1/p]$ and $\bplus_K=\aplus_K[1/p]$. When
$K=K_0$, this definition of $\aa_{K_0}$ coincides with the one given
above.

We define $\D(V)=(\bb\otimes V)^{H_{K}}$. It is a $d$-dimensional $\bb_{K}$-vector space with
Frobenius $\varphi$ and $\Gamma$-action. Similarly, if $T$ is a lattice of $V$, we define
$\D(T)=(\aa\otimes T)^{H_{K}}$, which is a free $\aa_{K}$-module of rank $d$.
We say a $\m$-module $D$ over $\bb_K$ is \emph{\'etale} if there is a free $\aa_K$-submodule $T$ of $D$,
which is stable under $\varphi$ and $\Gamma$ actions, such that $T\otimes_{\aa_K}\bb_{K}=D$.
Then $\D(V)$ is an \'etale $\m$-module since $D(T)$ is such an $\aa_K$-lattice.
The following result is due to Fontaine \cite{Fo91}.
\begin{theorem}
The functor $V\mapsto\D(V)$ is an equivalence from the category of $p$-adic
representations of $G_{K}$ to the category of \'etale
$(\varphi,\Gamma)$-modules over $\bb_{K}$; the inverse functor is
$D\mapsto(\bb\otimes D)^{\varphi=1}$.
\end{theorem}
 We define the ring of overconvergent elements as follows:
\begin{center}
 $\btdag{,r}=\displaystyle{\{x=\sum_{k>-\infty}^{+\infty}}p^{k}[x_{k}]\in \bt,
\displaystyle{\lim_{k\rightarrow+\infty} v_{\et}(x_{k})+kpr/(p-1)=+\infty \}}$
\end{center}
and $\btdag{}=\cup_{r\geqslant 0}\btdag{,r}$, $\bdag{,r}=(\btdag{,r})\cap\bb$,
$\bdag{}=\cup_{r\geqslant 0}\bdag{,r}$. Note that $\varphi:\btdag{,r}\rightarrow \btdag{,pr}$ is an isomorphism.
Let $\atdag{,r}$ be the set of elements of $\btdag{,r}\cap \at$ such that $v_{\et}(x_{k})+kpr/(p-1)\geq 0$
for every $k$ and similarly $\atdag{}=\cup_{r\geqslant 0}\atdag{,r}$, $\adag{,r}=\atdag{,r}\cap\aa$,
$\adag{}=\cup_{r\geqslant 0}\adag{,r}$.

Define $\D^{\dagger,r}(V)=(\bdag{,r}\otimes V)^{H_{K}}$ and
$\D^{\dagger}(V)=\cup_{r\geqslant 0} \D^{\dagger,r}(V)=(\bdag{}\otimes V)^{H_K}$. Similarly, if $T$ is a lattice of $V$,
we define $\D^\dagger(T)=(\adag{}\otimes T)^{H_{K}}$. We say a $\m$-module $D$ over $\bdag{}_K$ is \emph{\'etale}
if there is a free $\adag{}_K$-submodule $T$ of $D$, which is stable under $\varphi$ and $\Gamma$ actions,
such that $T\otimes_{\adag{}_K}\bdag{}_{K}=D$. In \cite{CC98}, Cherbonnier and Colmez proved the following result.
\begin{theorem}
There exists an $r(V)$ such that
$\D(V)=\bb_{K}\otimes_{\bdag{,r}_{K}}\D^{\dagger,r}(V)$ if $r\geqslant r(V) $.
Equivalently, $\D^{\dagger}(V)$ is a $d$-dimensional \'etale $\m$-module over
$\bdag{}_{K}$. Therefore, $V\mapsto\D^\dagger(V)$ is an equivalence from the category of $p$-adic
representations of $G_{K}$ to the category of \'etale
$(\varphi,\Gamma)$-modules over $\bdag{}_{K}$.
\end{theorem}
We can take $\pi_K$ to be an element of $\adag{}_K$ whose image modulo $p$ is $\overline{\pi}_K$.
Let $e_K$ denote the ramification index of $K_{\infty}/(K_0)_{\infty}$. Then for $r\gg 0$, one can show that
$\bdag{,r}_{K}$ is given by
\begin{center}
$\bdag{,r}_{K}=\{f(\pi_K)=\displaystyle{\sum^{+\infty}_{k=-\infty}}a_k\pi_K^k,$ where $a_k\in K_0^{'}$
and $f(T)$ is convergent and bounded on $p^{-1/e_Kr}\leq |T|< 1 \}$.
\end{center}
We see that the sup norms on closed annuli give a family of norms on $\bdag{,r}_K$. Its Fr\'echet completion with respect to these norms is
\begin{center}
$\btrig{,r}{,K}=\{f(\pi_K)=\displaystyle{\sum^{+\infty}_{k=-\infty}}a_k\pi_K^k,$ where $a_k\in K_0^{'}$
and $f(T)$ is convergent on $p^{-1/e_Kr}\leq |T|< 1 \}$.

\end{center}
Then the union $\btrig{}{,K}=\cup_{r\geq 0}\btrig{,r}{,K}$ can be
identified with the \emph{Robba ring} $\r_{K}$ from the theory of $p$-adic
differential equations, which is the set of holomorphic functions on
the boundary of the open unit disk, by mapping $\pi_K$ to $T$. And
$\bdag{}_K=\calE_K^{\dagger}$ is the subring of $\r_K$ consisting of
bounded functions. The $p$-adic completion of $\calE_K^{\dagger}$ is
$\calE_K=\bb_K$. When $K=K_0$, we can choose
$\bar{\pi}_{K}=\bar{\pi}=\varepsilon-1$ and
$\pi_{K}=\pi=[\varepsilon]-1$.

For any $K$, there exists an $r(K)$ such that for any $r\geq r(K)$
and $n\geq n(r)=(\log(r/(p-1))/\log p)+1$, we have an injective
morphism $\iota_n$ from $\btrig{,r}{,K}$ to $K_n[[t]]$ which
satisfies $\iota_{n}=\iota_{n+1}\circ \varphi$ (see Chapter 2 of
\cite{LB02} for the construction). For example, when $K=K_0$,
$\iota_n$ is defined by $\iota_n(\pi)=\eps^{(n)}e^{t/p^{n}}-1$.

Define the operator $\nabla=\log(\gamma)/\log(\chi(\gamma))$ which gives an action of Lie$(\Gamma_K)$ on $\r_K$.
Let $t=\log([\varepsilon])$ and set the differential operator $\partial=\nabla/t$ which satisfies:
\begin{center}
$\partial\circ\varphi=p\varphi\circ\partial$ and $\partial\circ\gamma=\chi(\gamma)\gamma\circ\partial$.
\end{center}
In case $K=\Qp$, we choose $\bar{\pi}_{K}=\bar{\pi}=\varepsilon-1$ and $\pi_{K}=\pi=[\varepsilon]-1$.
Then we have $\nabla(f(\pi))=\log(1+\pi)(1+\pi)df/d\pi$ and $\partial=(1+\pi)df/d\pi$.

Define
$\D^{\dagger,r}_{\rig}(V)=\D^{\dagger,r}(V)\otimes_{\bdag{,r}_{K}}\btrig{,r}{,K}
$ and $\D^\dagger_\rig(V)=\cup_{r\geqslant0}
\D^{\dagger,r}_{\rig}(V)
=\D^\dagger(V)\otimes_{\bdag{}_K}\btrig{}{,K}$ which is an \'etale
$\m$-module over $\btrig{}{,K}$. Here we say a $\m$-module $D$ over
$\btrig{}{,K}$ is \emph{\'etale} if $D$ has a $\bdag{}_K$-submodule
$D'$, which is an \'etale $\m$-module over $\bdag{}_K$ under the
restricted $\varphi$ and $\Gamma$ actions, such that
$D=D'\otimes_{\bdag{}_K}\btrig{}{,K}$. The following theorem is due
to Kedlaya (\cite{Ke06}).
\begin{theorem}
The correspondence $D\mapsto \btrig{}{,K}\otimes_{\bdag{}_K}D$ is an equivalence between the category of \'etale $\m$-modules over $\bdag{}_K$ and the category of \'etale $\m$-modules over $\btrig{}{,K}$. As a consequence, $V\mapsto\D^{\dagger}_{\rm rig}(V)$ is an equivalence of
categories from the category of $p$-adic representations of $G_{K}$ to the category of \'etale
$(\varphi,\Gamma)$-modules over $\btrig{}{,K}$.
\end{theorem}
Suppose $D$ is an arbitrary $\m$-module over $\btrig{}{,K}$ of $\rank d$. By a result of Berger \cite[Theorem I.3.3]{LB04}, for $r$ large enough, there is a unique $\btrig{,r}{,K}$ submodule $D_r$ of $D$ such that:
\begin{enumerate}
\item[(1)]$D_r$ is a free $\btrig{,r}{,K}$-module of rank $d$, stable under $\Gamma$ action, and
$D_r\otimes_{\btrig{,r}{,K}}\btrig{}{,K}=D$;
\item[(2)]We can find a $\btrig{,pr}{,K}$-basis of $D_r\otimes_{\btrig{,r}{,K}}\btrig{,pr}{,K}$ from
elements of $\varphi(D_r)$.
\end{enumerate}
If $D_r$ is defined, then for any $r'\geq r$ we have $D_{r'}=D_r\otimes_{\btrig{,r}{,K}}\btrig{,r'}{,K}$.
We set
\[
\D^{+,n}_{\dif}(D)=D_r\otimes_{\btrig{,r}{,K},\iota_n}K_n[[t]]
\]
and call it the \emph{localization at $\eps^{(n)}-1$} of $D$. It is easy to see that $\D^{+,n}_{\dif}(D)$ is well defined,
i.e. it does not depend on the choice of $r$. Let $i_n$ denote the natural inclusion map from
$K_n[[t]]$ to $K_{n+1}[[t]]$, and define the \emph{connecting map} $\varphi_n$ as
\[
\varphi\otimes i_n:\D^{+,n}_{\dif}(D)\longrightarrow\D^{+,n+1}_{\dif}(D).
\]
It is clear that
\[
\varphi_n\otimes 1:\D^{+,n}_{\dif}(D)\otimes_{K_n[[t]]}K_{n+1}[[t]]\longrightarrow\D^{+,n+1}_{\dif}(D)
\]
is an isomorphism.
\subsection{Slope theory of $\varphi$-modules}
This section is a short collection of some basic facts concerning the slope theory of
$\varphi$-modules over $\r_K$ (resp. $\calE^\dagger_K$) which we will use later. For a complete treatment of this topic, see \cite{Ke06}.

A \emph{$\varphi$-module} $M$ over $\r_K$ (resp. $\calE^\dagger_K$)
is a finitely generated free $\r_K$-module (resp.
$\calE^\dagger_K$-module) with a Frobenius action $\varphi$ that
satisfies $\varphi^*M\cong M$. We can view $M$ as a left module over
the twisted polynomial ring $\r_K\{X\}$. For a positive integer $a$,
define the \emph{a-pushforward} functor $[a]_*$ from
$\varphi$-modules to $\varphi^a$-modules to be the restriction along
the inclusion $\r_K\{X^a\}\ra \r_K\{X\}$. Define the
\emph{a-pullback} functor $[a]^*$ from $\varphi^a$-modules to
$\varphi$-modules to be the extension of scalars functor $M\ra
M\otimes_{\r_K\{X^a\}}\r_K\{X\}$. If $\rank M=n$, then $\wedge^n M$
has rank $1$ over $\r_K$. Let $v$ be a generator, then
$\varphi(v)=\lambda v$ for some $\lambda\in
\r_K^{\times}=(\calE_K^\dagger)^{\times}$. Define the \emph{degree}
of $M$ by setting $\deg(M)=w(\lambda)$, here $w$ is the $p$-adic
valuation of $\calE_K$. Note that this does not depend on the choice
of $v$. If $M$ is nonzero, define the \emph{slope} of $M$ by setting
$\mu(M)=\deg(M)/\rank(M)$. The following formal properties are
easily verified:
\begin{enumerate}
\item[(1)] If $0\rightarrow M_1\rightarrow M\rightarrow
M_2\rightarrow 0$ is exact, then $\deg(M)=\deg(M_1)+\deg(M_2)$;
\item[(2)] We have $\mu(M_1\otimes M_2)=\mu(M_1)+\mu(M_2)$;
\item[(3)] We have $\deg(M^\vee)=-\deg(M)$ and $\mu(M^\vee)=-\mu(M)$.
\end{enumerate}

We say $M$ is \emph{\'etale} if $M$ has a $\varphi$-stable
$\OO_{\calE^\dagger_K}$-submodule $M'$ such that $\varphi^*M'\cong
M'$ and $M'\otimes_{\r^{int}_K}\r_K=M$. More generally, suppose
$\mu(M)=s=c/d$, where $c,d$ are coprime integers with $d>0$. We say
$M$ is \emph{pure} if for some $\varphi$-module $N$ of rank $1$ and
degree $-c$, $([d]_*M)\otimes N$ is \'etale. We say a $\m$-module
over $\r_K$ (resp. $\calE_K^\dagger$) is \emph{pure} if the
underlying $\varphi$-module structure is pure. In the \'etale case,
this definition is consistent with the one given in the last
section. We have the following facts:
\begin{enumerate}
\item[(1)] A $\varphi$-module is pure of slope $0$ if and only if it is
\'etale;
\item[(2)] The dual of a pure $\varphi$-module of slope $s$ is itself
pure of slope $-s$;
\item[(3)] If $M_1$, $M_2$ are pure of slopes $s_1$, $s_2$, then
$M_1\otimes M_2$ is pure of slope $s_1+s_2$.
\end{enumerate}

We say $M$ is \emph{semistable} if for every nontrivial
$\varphi$-submodule $N$, we have $\mu(N)\geq \mu(M)$. A difficult
result is that $M$ is semistable if and only if it is pure \cite[Theorem 2.1.8]{Ke06}. As a consequence of this result, we have the following slope filtrations theorem.
\begin{theorem}(Kedlaya)
Every $\varphi$-module $M$ over $\r_K$ admits a unique filtration
$0=M_0\subset M_1\subset...\subset M_l=M$ by saturated
$\varphi$-submodules whose successive quotients are pure with
$\mu(M_1/M_0)<...<\mu(M_{l-1}/M_l)$. As a consequence, if $M$ is a
$\m$-module, then these $M_i$'s are all $\m$-submodules.
\end{theorem}
\section{Cohomology of $\m$-modules}

\subsection{Construction of cohomology}
Suppose $D$ is a $\m$-module over $\calE_K^\dagger$, $\calE_K$, or $\r_K$. Let $\triangle_K$ be a torsion subgroup of $\Gamma_K$. Since $\Gamma_{K}$ is an open subgroup of $\Z^{\times}$,
$\triangle_K$ is a finite group of order dividing $p-1$ (or 2 if
$p=2$). Let $p_{\triangle}$ be the idempotent operator defined by
$p_{\triangle}=(1/|\triangle_K|)\sum_{\delta\in\triangle_K}\delta$.
Then $p_{\triangle}$ is the projection from $D$ to
$D'=D^{\triangle_K}$. In case $\Gamma_K/\triangle_K$ is procyclic, we set the following
complex, where $\gamma$ denotes a topological generator of $\Gamma_K$:
 \[
  C^{\bullet}_{\varphi,\gamma}(D):
 0\stackrel{}{\longrightarrow}D'\stackrel{d_{1}}{\longrightarrow}D'\oplus D'
 \stackrel{d_{2}}{\longrightarrow}D'\stackrel{}{\longrightarrow}0,
\]
with $d_1(x) = ((\gamma - 1)x, (\varphi - 1)x)$ and $d_2(x) =
((\varphi - 1)x - (\gamma - 1)y)$. Let $H^{\bullet}(D)$ denote
cohomology groups of this complex. We need to check $H^{\bullet}(D)$
is well defined, i.e. it does not depend on the choice of
$\triangle_K$. In the following, we assume $D$ is over
$\calE_K^\dagger$. The argument also works for $\m$-modules over
$\calE_K$ and $\r_K$.

First, it is obvious that $H^0(D)=D^{\Gamma=1, \varphi=1}$. For
$H^1$, we claim $H^1(D)$ classifies all the extensions of
$\calE_K^\dagger$ by $D$ in the category of $\m$-modules over
$\calE^\dagger_K$. In fact, if $D$ is a $\m$-module over
$\calE^\dagger_K$ and $D_1$ is a such extension, we get the
following commutative diagram:
\[
\xymatrix{0 \ar[r] &
D\ar^{p_{\triangle}}[d] \ar[r] & D_1
\ar^{p_{\triangle}}[d] \ar[r]
&\calE^\dagger_K \ar^{p_\triangle}[d] \ar[r] & 0 \\
0\ar[r] & D'
\ar[r] & (D_1)' \ar[r] & (\calE^\dagger_K)'
\ar[r] & 0 .}
\]
Since $| \triangle_K|$ divides $p-1$ (or $2$ if $p$=2), all the characters of $\triangle_K$
take values in $\Qp\subset\calE^\dagger_K$. Then by standard representation theory, we have the eigenspace
decomposition $D=\oplus_{\chi}D_{\chi}$ for any $D$. Here $\chi$ ranges over all the characters of $\triangle_K$,
and $D_\chi$ is the $\chi$-eigenspace. Any nonzero element $x$ of $(\calE^\dagger_K)_\chi$ (e.g.
$\sum_{\delta\in\triangle_K}\chi(\delta^{-1})\delta (\eps)$) gives an isomorphism $D'\cong D_\chi$ by mapping $a$
to $xa$. Therefore we have $D'\otimes_{(\calE^\dagger_K)'}\calE^\dagger_K\cong D$, where the isomorphism respects $\varphi$ and $\Gamma_K$-actions. So the extensions of $\calE_K^\dagger$ by $D$ as $(\varphi, \Gamma_K)$-modules over $\calE^\dagger_K$ are in one-to-one correspondence with the extensions of $(\calE^\dagger_K)'$ by $(D_1)'$ as $(\varphi, \Gamma_K/\triangle_K)$-modules over $(\calE^\dagger_K)'$. The latter objects are clearly classified by $H^1(C^{\bullet}_{\varphi,\gamma}(D))$.

For $H^2$, suppose $\triangle'_K\supset\triangle_K$ is another torsion subgroup of $\Gamma$ and
$m=[\triangle'_K:\triangle_K]$. Then $\gamma^m$ is a topological generator of $\Gamma/\triangle'_K$
and $p_{\triangle/\triangle'}=1/m\sum_{i=0}^{m-1}\gamma^i$ is a projection from $D^{\triangle_K}$ to $D^{\triangle'_K}$.
Obviously $p_{\triangle/\triangle'}$ reduces to a projection
\[
D^{\triangle_K}/(\gamma-1)\stackrel{p_{\triangle/\triangle'}}
{\longrightarrow} D^{\triangle'_K}/(\gamma^m-1).
\]
Similarly as above, we have $D^{\triangle_K}=D^{\triangle'_K}\oplus(\oplus_{\chi\neq1}D^{\triangle_K}_{\chi})$ where
$\chi$ ranges over all the non-trivial characters of $\triangle'_K/\triangle_K$. Note that $\gamma-1$ acts bijectively
on any $D^{\triangle_K}_\chi$ with $\chi\neq1$, so the natural map
\[
D^{\triangle'_K}/(\gamma^m-1)\stackrel{i}
{\longrightarrow} D^{\triangle_K}/(\gamma-1)
\]
is surjective. We conclude that both
$i$ and $p_{\triangle/\triangle'}$ are isomorphisms. So there are canonical isomorphisms between $H^2(D)$'s
respecting different choices of torsion subgroups.

Finally, we define cup products as follows:
\[
H^0(M)\times H^{0}(N) \rightarrow H^{0}(M\otimes N) \qquad(x,y)\mapsto x\otimes y
\]
\[
 H^{0}(M)\times H^{1}(N) \rightarrow H^{1}(M\otimes N) \qquad
(x,(\bar{y},\bar{z}))\mapsto(\overline{x\otimes y}, \overline{x\otimes z})
\]
\[
 H^{0}(M)\times H^{2}(N)\rightarrow H^{2}(M\otimes N) \qquad (x,\bar{y})\mapsto \overline{x\otimes y}
\]
\[
H^{1}(M)\times H^{1}(N)\rightarrow H^{2}(M\otimes N) \qquad ((\bar{x},\bar{y}),(\bar{z},\bar{t}))
\mapsto\overline{y\otimes\gamma(z)-x\otimes\varphi(t)}
\]

\subsection{Shapiro's lemma}

If $L$ is a finite extension of $K$, and $D$ is a $\m$-module over
$\calE_{L}$  (resp. $\calE^\dagger_{L}$, $\r_{L}$). Consider
$\Ind^{\Gamma_K}_{\Gamma_L}D=\{f:\Gamma_K\ra D|f(hg)=hf(g)$ for
$h\in\Gamma_L\}$, the induced $\Gamma_K$-representation of $D$ as a
$\Gamma_L$-representation. We can endow
$\Ind^{\Gamma_K}_{\Gamma_L}D$ with an $\calE_{K}$ (resp.
$\calE^\dagger_{K}$, $\r_{K}$) module structure and a Frobenius
action $\varphi$ by defining $(af)(g)=g(a)f(g)$ and
$(\varphi(f))(g)=\varphi(f(g))$ for any element
$f:\Gamma_K\rightarrow D$ of $\Ind^{\Gamma_K}_{\Gamma_L}D$ and
$g\in\Gamma_K$. In this way $\Ind^{\Gamma_{K}}_{\Gamma_{L}}D$ is
now a $\m$-module over $\calE_{K}$  (resp. $\calE^\dagger_{K}$,
$\r_K$). We call it the \emph{induced $\m$-module of $D$} from $L$
to $K$, and denote it by $\Ind_{L}^{K}D$. Note that
$\rank\Ind^K_{L}D=[L:K]\rank D$.

One can prove that the above definition of induced $\m$-modules is compatible with the definition of induced representations of Galois representations.
\begin{proposition}
Suppose $V$ is a $p$-adic representation of $G_{L}$, then $\D(\Ind^{G_K}_{G_{L}}V)=\Ind^K_{L}\D(V)$ (resp. $\D^\dagger$, $\D^\dagger_\rig$).
\end{proposition}
\begin{proof}
For the functor $\D$, we define a map $P$ from
$\D(\Ind^{G_K}_{G_{L}}V)=((\Ind^{G_K}_{G_{L}}V)\otimes_{\Qp}\bb)^{H_K}$
to $\Ind^K_{L}\D(V)$ as follows: for $\sum f_i\otimes b_i\in
((\Ind^{G_K}_{G_{L}}V)\otimes_{\Qp}\bb)^{H_K}$ and $\bar{g}\in
\Gamma_K$, we put $P(\sum f_i\otimes b_i)(\bar{g})=\sum f_i(g)\otimes g
b_i$, where $g$ is any lift of $\bar{g}$ in $G_K$. To see that $P$ is
well defined,  we first need to show that it doesn't depend on the
choice of $g$. In fact, for any $h\in H_K$ we have
\begin{center}$\sum f_i(gh)\otimes gh b_i=\sum (hf_i)(g)\otimes
g(hb_i)=\sum f_i(g)\otimes g b_i$,
\end{center}
where the last equality is concluded from the fact that $\sum
f_i\otimes b_i=\sum hf_i\otimes hb_i$, since $\sum f_i\otimes b_i$ is
$H_K$-invariant. Then for $h\in H_L$, we have
\[
h(\sum f_i(g)\otimes g b_i)=\sum hf_i(g)\otimes hgb_i=\sum
f_i(hg)\otimes hgb_i=\sum f_i(g)\otimes g b_i,
\]
since $hg$ is also a lift of $g$. This implies that $P(\sum
f_i\otimes b_i)(\bar{g})$ lies in $\D(V)$. For $\bar{h}\in
\Gamma_L$, we have
\[
P(\sum f_i\otimes b_i)(\bar{h}\bar{g})=\sum f_i(hg)\otimes
hgb_i=h(\sum f_i(g)\otimes gb_i)=\bar{h}(P(\sum f_i\otimes
b_i)(\bar{g}).
\]
It follows that $P(\sum f_i\otimes b_i)$ really lies in
$\Ind^K_{L}\D(V)$. It is obvious that $P$ is injective and
commutes with $\varphi$. Now we check that $P$ is a morphism of
$\m$-modules. For $a\in \calE_K$, we have
\[P(a(\sum f_i\otimes b_i))(\bar{g})=P(\sum f_i\otimes
ab_i)(\bar{g})=\sum f_i(g)\otimes g(a)g(b_i)=g(a)(\sum f_i(g)\otimes
gb_i)=(a(P(\sum f_i\otimes b_i)))(\bar{g}).
\] So $P$ is a morphism of $\calE_K$-modules. For $\bar{h}\in \Gamma_K$, we have
\[
P(\bar{h}(\sum f_i\otimes b_i))(\bar{g})=P(\sum hf_i\otimes
hb_i))(\bar{g})=\sum (hf_i)(g)\otimes ghb_i=\sum f_i(gh)\otimes
ghb_i=(\bar{h}P(\sum f_i\otimes b_i))(\bar{g}),
\]
hence $P$ is $\Gamma_K$-equivariant. Now note that
\begin{center}
$\dim_{\calE_K}(\Ind_{L}^{K}\D(V))=[L:K]\dim_{\calE_L}\D(V)=[L:K]\dim
V=\dim_{\calE_K}\D(\Ind^{G_K}_{G_{L}}V)$,
\end{center}
so $P$ is an isomorphism. Since $\Ind^K_L\D^\dagger(V)$ is an
\'etale $\m$-module over $\calE^\dagger_K$ contained in
$\Ind^K_{L}\D(V)=\D(\Ind^{G_K}_{G_{L}}V)$ and of maximal
dimension, we conclude that
$\Ind^K_{L}\D^\dagger(V)=\D^\dagger(\Ind^{G_K}_{G_{L}}V)$. Finally,
we get
\begin{center}
$\Ind^K_{L}\D^\dagger_{\rig}(V)=\Ind^K_{L}\D^\dagger(V)\otimes
\r_K=\D^\dagger(\Ind^{G_K}_{G_{L}}V)\otimes
\r_K=\D^\dagger_{\rig}(\Ind^{G_K}_{G_{L}}V)$.
\end{center}

\end{proof}
\begin{theorem}(Shapiro's Lemma for $\m$-modules)
Suppose $D$ is a $\m$-module over $\calE_{L}$, $\calE^\dagger_{L}$
or $\r_L$. Then there are isomorphisms
\[
H^i(D)\cong H^i(\Ind_{L}^KD)   \qquad (i=0,1,2)
\]
which are functorial in $D$ and compatible with cup products.

\end{theorem}
\begin{proof}We first prove the theorem in the case that both of $\Gamma_K$ and $\Gamma_{L}$ are procyclic.
Suppose $[\Gamma_K:\Gamma_{L}]=m$. Choose a topological generator
$\gamma_K$ of $\Gamma_K$, then $\gamma_L=\gamma_K^m$ is a
topological generator of $\Gamma_{L}$. Define
$Q:D\rightarrow\Ind_{L}^KD$ as follows: for any $x\in D$,
$(Q(x))(e)=x$ and $(Q(x))(\gamma_K^i)=0$ for $1\leq i\leq m-1$. Then
$Q$ is a well defined $\varphi$, $\Gamma_{L}$-equivariant injective
morphism of $\r_K$-modules. We claim that $Q$ induces a
$\varphi$-equivariant isomorphism from $D/(\gamma_L-1)$ to
$(\Ind^K_{L}D)/(\gamma_K-1)$. Suppose $x\in D$, and
$Q(x)=(\gamma_K-1)f$ for some $f\in\Ind^K_{L}D$. Then we have
\begin{eqnarray*}
 x&=&Q(x)(e)=(\gamma_K-1)f(e)=f(\gamma_K)-f(e) \\
 0&=&Q(x)(\gamma_K^i)=(\gamma_K-1)f(\gamma_K^i)=f(\gamma_K^{i+1})-f(\gamma_K^i) \qquad 1\leq i\leq m-1. \end{eqnarray*}
Summing these equalities, we get $x=\sum_{i=0}^{m-1}(f(\gamma_K^{i+1})-f(\gamma_K^i))=f(\gamma_K^{m})-f(e)=(\gamma_L-1)f(e)$ since $\gamma_K^m=\gamma_L$.
 On the other hand, for any $f\in\Ind^K_{L}D$, suppose $f(\gamma_K^i)=x_i$ for $0\leq i\leq m-1$.
 Then $f=\sum_{i=1}^{m}\gamma_K^i Q((\gamma_L)^{-1}x_{m-i})$ since for $0\leq j\leq m-1$, we have
 \begin{center}
 $(\sum_{i=1}^{m}\gamma_K^iQ((\gamma_L)^{-1}x_{m-i}))(\gamma_K^j)
 =\sum_{i=1}^{m}Q((\gamma_L)^{-1}x_{m-i})(\gamma_K^{i+j})=Q((\gamma_L)^{-1}x_{j})(\gamma_K^m)=x_j$.
 \end{center}
 So both of $f$ and $Q(x)$, where $x=\gamma_L^{-1}(\sum_{i=1}^{m}x_{m-i})$, have the same image in $(\Ind^K_{L}D)/(\gamma_K-1)$.

For any $g\in\Gamma_K$, define the morphism $Q^g$ by
$Q^g(x)=g(Q(x))$ for any $x\in D$. Set
$\widetilde{Q}=\sum_{i=0}^{m-1}Q^{\gamma_K^i}$ which is also
$\varphi$, $\Gamma_{L}$-equivariant and injective since
$(\widetilde{Q}(x))(e)=x$. We claim that $\widetilde{Q}$ induces an
$\varphi$-equivariant isomorphism from $D^{\Gamma_{L}}$ to
$(\Ind^K_{L}D)^{\Gamma_K}$. The injectivity is obvious. Conversely,
suppose $f:\Gamma_K\rightarrow D$ is an element of
$(\Ind^K_{L}D)^{\Gamma_K}$ with $f(e)=x$. Then $f(g)=(gf)(e)=f(e)=x$
for any $g\in\Gamma_K$ since $f$ is $\Gamma_K$-invariant. On the
other hand, for $g\in\Gamma_{L}$, we have $f(g)=gf(e)=gx$. These
imply that $x$ is $\Gamma_{L}$-invariant. Therefore
$\widetilde{Q}(x)=f$.

Consider the following commutative diagram:
\[
\xymatrix{ C^{\bullet}_{\varphi,\gamma'}(D):0 \ar[r] &
D\ar^{\widetilde{Q}}[d] \ar[r] & D\oplus D
\ar^{Q\oplus\widetilde{Q}}[d] \ar[r]
& D \ar^{Q}[d] \ar[r] & 0 \\
C^{\bullet}_{\varphi,\gamma}(\Ind^K_{L}D):0\ar[r] & \Ind^K_{L}D
\ar[r] & \Ind^K_{L}D\oplus \Ind^K_{L}D \ar[r] & \Ind^K_{L}D \ar[r] &
0. }
\]
This induces morphisms $\alpha^i$ from $H^{i}(D)$ to
$H^{i}(\Ind^K_{L}D)$. We will prove that they are isomorphisms.

For $H^0$,  $\widetilde{Q}$ induces a $\varphi$-equivariant
isomorphism from $D^{\Gamma_{L}}$ to $(\Ind^K_{L}D)^{\Gamma_K}$.
Taking $\varphi$-invariants, we conclude that $\alpha^0$ is an
isomorphism. For $H^2$, $Q$ induces a $\varphi$-equivariant
isomorphism from $D/(\gamma_L-1)$ to $(\Ind^K_{L}D)/(\gamma_K-1)$,
so $\alpha^2$ is also an isomorphism.

For $H^1$, we use the following commutative diagram:
\[
\xymatrix{0 \ar[r] &
D^{\Gamma_{L}}/(\varphi-1)\ar^{\widetilde{Q}}[d] \ar[r] & H^1(D)
\ar^{\alpha^1}[d] \ar[r]
& (D/(\gamma_L-1))^{\varphi=1} \ar^{Q}[d] \ar[r] & 0 \\
0\ar[r] & (\Ind^K_{L}D)^{\Gamma_K}/(\varphi-1)
\ar[r] & H^1(\Ind^K_{L}D)\ar[r] & (\Ind^K_{L}D/(\gamma_K-1))^{\varphi=1}
\ar[r] & 0. }
\]
We have proved that $\widetilde{Q}$ and $Q$ are isomorphisms. So $\alpha^1$ is an isomorphism by the Five Lemma.

For the general case, let $\triangle_K$ and $\triangle_{L}$ be the
torsion subgroups of $\Gamma_K$ and $\Gamma_{L}$ respectively. Then
$\Gamma_{L}/\triangle_{L}$ is a subgroup of $\Gamma_K/\triangle_K$.
Let $\gamma_K$ be a topological generator of $\Gamma_K/\triangle_K$.
Suppose $[\Gamma_K/\triangle_K:\Gamma_{L}/\triangle_{L}]=m$, then
$\gamma_L=\gamma_K^m$ is a topological generator of
$\Gamma_{L}/\triangle_{L}$. Set $Q':D'\rightarrow(\Ind_{L}^KD)'$ as
follows: for any $x\in D'$, $(Q'(x))(e)=x$ and $(Q'(x))(y)=0$ for
any other representative of $\Gamma_K/\Gamma_{L}$. We define
$\widetilde{Q}^{'}=\sum_{i=0}^{m-1}\gamma^iQ'$. Replacing $Q$ by
$Q'$, and $\widetilde{Q}$ by $\widetilde{Q}^{'}$ in the above
argument, we are done.
\end{proof}

\subsection{Comparison theorems}
For a $\Z$-representation $V$ (of finite length or not), define $H^\bullet(\D(V))$ using the same complex as in the last section. The groups $H^\bullet(\D(V))$ are also well defined by the same argument (Note: for $p=2$, there is only one choice of $\triangle_K$, so it is well defined automatically. However, the description of $H^1$ in terms of extensions does not apply to $\mathbb{Z}_2$, because the projection $p_\Delta$ is not integral.). The following theorem was first proved by Herr (\cite{H98}) in case $\Gamma_K$ is procyclic. Our result is a small improvement of his result since $\Gamma$ is always procyclic for $p\neq2$.
\begin{theorem}Let $V$ be a $\Z$-representation of $G_K$. Then there are
isomorphisms
\[
H^i(\D(V)) \cong H^i(G_K, V) \qquad (i=0,1,2)
\]
which are functorial in $V$ and
compatible with cup products. The same conclusion therefore also holds for $\Qp$-representations.
\end{theorem}
\begin{proof}
For $V$ of finite length, we adapt the proof given by \cite[Theorem 5.2.2]{C04} to the case, where $\Gamma$ is not necessarily procyclic. Let $H'_K$ denote the preimage of $\triangle_K$ in $G_K$. Replacing $H_K$ by $H'_K$ and
$\D(V)$ by $(\D(V))'$ in their proof then it works for general $\Gamma$.
For general $V$, note that the inverse system $\{H^i(\D(V/p^nV))\cong H^i(G_K, V/p^nV)\}$
satisfies the Mittag-Leffler condition,
so we can conclude the result by taking the inverse limit of $\{H^i(\D(V/p^nV))\cong H^i(G_K, V/p^nV)\}$.
\end{proof}
In the remainder of this section, $V$ is a $\Qp$-representation.
\begin{lemma}
The morphism $\gamma-1:((\D^\dagger(V))')^{\psi=0}\longrightarrow((\D^\dagger(V))')^{\psi=0}$ has a continuous inverse.
\end{lemma}
\begin{proof}
Note that $\chi(\Gamma_{K_1})\subset 1+p\Z$ is procyclic. We can
choose a topological generator $\gamma'$ of $\Gamma_{K_1}$ such that
$\gamma'=\gamma^m$ in $\Gamma/\triangle_K$ for some
$m\in\mathbb{N}$. Consider the commutative diagram:
\[
\xymatrix{
\D^\dagger(V)^{\psi=0}\ar^{p_{\triangle_K}}[d] \ar^{\gamma'-1}[r]& \D^\dagger(V)^{\psi=0} \ar^{p_{\triangle_K}}[d] \\
 ((\D^\dagger(V))')^{\psi=0} \ar^{\gamma^m-1}[r]&((\D^\dagger(V))')^{\psi=0}.}
\]
Since $\D^\dagger(V)^{\psi=0}\stackrel{\gamma'-1}\longrightarrow
\D^\dagger(V)^{\psi=0}$ has a continuous inverse by \cite[Proposition 2.6.1]{CC98}, and $p_{\triangle_K}$ is an idempotent operator,
we get that $((\D^\dagger(V))')^{\psi=0} \stackrel{\gamma^m-1}\longrightarrow((\D^\dagger(V))')^{\psi=0}$
has a continuous inverse. Then $(\gamma-1)^{-1}=(\gamma^m-1)^{-1}(1+\gamma+...+\gamma^{m-1})$ is also continuous.
\end{proof}
\begin{lemma}
Let $C^{\bullet}_{\psi,\gamma}(\D^\dagger(V))$ be the complex
\[
  0\stackrel{}{\longrightarrow}(\D^\dagger(V))'\stackrel{d_{1}}{\longrightarrow}(\D^\dagger(V))'\oplus (\D^\dagger(V))'
 \stackrel{d_{2}}{\longrightarrow}(\D^\dagger(V))'\stackrel{}{\longrightarrow}0
\]
with $d_1(x) = ((\gamma - 1)x, (\psi - 1)x)$ and $d_2((x,y)) =
((\psi - 1)x - (\gamma - 1)y)$. Then we have a commutative diagram
of complexes
\[
\xymatrix{ C^{\bullet}_{\psi,\gamma}(\D^\dagger(V)):0 \ar[r] &
(\D^\dagger(V))'\ar^{id}[d] \ar[r] &
(\D^\dagger(V))'\oplus(\D^\dagger(V))' \ar^{-\psi\oplus id}[d]
\ar[r]
& (\D^\dagger(V))' \ar^{-\psi}[d] \ar[r] & 0 \\
C^{\bullet}_{\psi,\gamma}(\D^\dagger(V)):0\ar[r] & (\D^\dagger(V))'
\ar[r] & (\D^\dagger(V))'\oplus (\D^\dagger(V))' \ar[r] &
(\D^\dagger(V))' \ar[r] & 0 }
\]
which induces an isomorphism on cohomology.
\end{lemma}
\begin{proof}
Since $\psi$ is surjective, the cokernel complex is $0$. The kernel
complex is
\[
0\longrightarrow0\longrightarrow ((\D^\dagger(V))')^{\psi=0}\stackrel{\gamma-1}\longrightarrow
((\D^\dagger(V))')^{\psi=0}\longrightarrow 0.
\]
 which has trivial cohomology by Lemma 2.2.
\end{proof}

\begin{lemma}
 Let $T$ be a $G_K$-stable $\Z$-lattice of $V$. Then the natural morphism $\D^\dagger(T)/(\psi-1)\rightarrow \D(T)/(\psi-1)$ is an isomorphism.
\end{lemma}
\begin{proof}
We can view  $\D(T)/(\psi-1)$ (resp. $\D^\dagger(T)$) as an \'etale $\varphi$-module over $\aa_{\Q}$ (resp. $\aa^\dagger_{\Q}$). For $x\in\aa_{\Q}$ and $n\in\mathrm{N}$, define $w_n(x)\in\mathrm{N}$ to be the smallest integer $k$ such that $x\in\pi^{-k}\aa_{\Q}+p^{n+1}\aa_{\Q}$. Short
computations show that $w_n(x+y)\leq \sup\{w_n(x), w_n(y)\}$, $w_n(xy)\leq w_n(x)+w_n(y)$ and $w_n(\varphi(x))\leq pw_n(x)$. By \cite[Proposition III 2.1]{CC99}, for any interger $m>1$, $x\in \aa^{\dagger,m}_{\Q}$  if and only if $w_n(x)-n(p-1)(p^{m-1}-1)\leq0$ for every $n$, and moreover approaches $-\infty$ as $n\rightarrow\infty$. For a vector or matrix $X$ with entries in $\aa_{\Q}$, define $w_n(X)$ as the maximal $w_n$ among the entries. Pick a basis $\{e_1, e_2,..., e_d\}$ of $\D^\dagger(T)$ over $\aa_{\Q}^\dagger$. For any $x\in\D(T)$, define
$w_n(x)=w_n(X)$ if $x=X(e_1, e_2,..., e_d)^t$. Let $A\in \rm{GL}(\aa_{\Q}^{\dagger})$ defined by $\varphi(e_1,e_2,..., e_d)^t=A(e_1, e_2,..., e_d)^t$.

Suppose $x=\psi(y)-y$, for $x=X(e_1, e_2,..., e_d)^t\in \D^\dagger(T)$ and $y=Y(e_1, e_2,..., e_d)^t\in\D(T)$. Then from \cite[Lemma I.6.4]{CC99} we have
\begin{eqnarray*}
w_n(y)&\leq&\max\{w_n(x), \frac{p}{p-1}(w_n(A^{-1})+1)\}.
\end{eqnarray*}
Now suppose all the entries of $X$ and $A^{-1}$ lie in $\aa_{\Q}^{\dagger,m}$ for some $m$. It follows that all the entries of $Y$ are in $\aa_{\Q}^{\dagger,m+1}$, hence $y\in\D^\dagger(T)$. This proves the injectivity of $\D^\dagger(T)/(\psi-1)\rightarrow \D(T)/(\psi-1)$.

Since $\D(T)/(\psi-1)$ is a finite $\Z$-module (\cite[Proposition 3.6]{H98}), so too is $\D^\dagger(T)/(\psi-1)$. Note that
\[
\D^\dagger(T)/(p)=\D^\dagger(T/(p))=\D(T/(p))=\D(T)/(p)
\]
since $\D^\dagger$ and $\D$ are identical at torsion level. Therefore
\[
 (\D^\dagger(T)/(\psi-1))/(p)=(\D^\dagger(T)/(p))/(\psi-1)=(\D(T)/(p))/(\psi-1)=(\D(T)/(\psi-1))/(p).
\]
This implies $\D^\dagger(T)/(\psi-1)\rightarrow\D(T)/(\psi-1)$ is
surjective by Nakayama's Lemma. Hence $\D^\dagger(T)/(\psi-1)\ra\D(T)/(\psi-1)$ is an isomorphism.
\end{proof}
\begin{proposition} Let $V$ be a $p$-adic representation of
$G_K$. Then the natural morphisms
\[
H^i(\D^\dagger(V))\stackrel{\alpha_i}\longrightarrow H^i(\D^\dagger_{\rig}(V)),
\quad H^i(\D^\dagger(V))\stackrel{\beta_i}\longrightarrow H^i(\D(V)) \qquad i=0,1,2 \]
are all isomorphisms which are functorial in $V$ and compatible with
cup products.
\end{proposition}

\begin{proof}
It is clear that these morphisms are functorial in $V$ and
compatible with cup products. To prove they are isomorphisms,
first note that $H^1(\D^\dagger(V))$ (resp.
$H^1(\D^\dagger_{\rig}(V)$, $H^1(\D(V))$) classifies all the
extensions of $\calE^\dagger_K$ (resp. $\r_K$, $\calE_K$) by
$\D^{\dagger}(V)$ (resp. $\D_{\rig}^{\dagger}(V), \D(V)$) in the
category of \'etale $\m$-modules over $\calE^\dagger_K$ (resp.
$\r_K$, $\calE_K$). Since these categories are all equivalent to the
category of $p$-adic representations by Theorems 1.1, 1.2 and 1.3,
we conclude that both $\alpha_1$, $\beta_1$ are isomorphisms.

From  \cite[Proposition 1.5.4]{Ke06}, the natural maps
$\D^\dagger(V)^{\varphi=1}\rightarrow
\D^\dagger_{\rig}(V)^{\varphi=1}$ and
$\D^\dagger(V)/(\varphi-1)\rightarrow\D^\dagger_{\rig}(V)/(\varphi-1)$
are bijective. Taking $\triangle_K$-invariants of the first map,  we have that $((\D^\dagger(V))')^{\varphi=1}\rightarrow
((\D^\dagger_{\rig}(V))')^{\varphi=1}$  is an isomorphism. As in Lemma 2.4, by the following commutative diagram
\[
\xymatrix{
\D^\dagger(V)/(\varphi-1)\ar^{p_{\triangle_K}}[d] \ar[r]& \D^\dagger_\rig(V)/(\varphi-1) \ar^{p_{\triangle_K}}[d] \\
 (\D^\dagger(V))'/(\varphi-1) \ar[r]&(\D^\dagger_\rig(V))'/(\varphi-1)}
\]
and the fact that $p_{\triangle_K}$ is an idempotent operator, we get
$(\D^\dagger(V))'/(\varphi-1)\rightarrow(\D^\dagger_{\rig}(V))'/(\varphi-1)$ is also an isomorphism.
Therefore $\alpha_0$ and $\alpha_2$ are isomorphisms.

Since $H^0(\D^\dagger(V))=V^{\Gamma_K}=H^0(\D(V))$, we conclude that $\beta_0$ is an isomorphism.
By Lemmas 2.5 and 2.6 we have
\[
H^2(\D^\dagger(V))=(\D^\dagger(V))'/(\psi-1,\gamma-1)=(\D(V))'/(\psi-1,\gamma-1)=H^2(\D(V)).
\]
Hence $\beta_2$ is an isomorphism.
\end{proof}
As a consequence of this proposition, there are canonical isomorphisms \[
 H^i(\D^\dagger_\rig(V))\stackrel{\beta_i\alpha_i^{-1}}\longrightarrow H^i(\D(V))  \qquad
 i=0,1,2.
 \]
Composing them with isomorphisms in Theorem 2.3, we get the
following theorem.
\begin{theorem}
Let $V$ be a $p$-adic representation of $G_K$. Then there are
isomorphisms
\[
H^i(\D^\dagger(V)) \cong H^i(G_K, V) \qquad (i=0,1,2)
\]
\[
H^i(\D^\dagger_{\rig}(V)) \cong H^i(G_K, V) \qquad
(i=0,1,2)
\]
which are functorial in $V$ and compatible with cup products.
\end{theorem}
\begin{corollary}The Euler-Poincar\'e characteristic formula and Tate local duality hold for all \'etale $\m$-modules over the Robba ring.
\end{corollary}
\begin{proof}
From the above theorem, we have that $H^2(\D^\dagger_\rig(\Qp(1)))$ is canonically isomorphic to $H^2(\Qp(1))$,
and then the Euler-Poincar\'e characteristic formula and Tate local duality for \'etale $\m$-modules follow from the usual Euler-Poincar\'e characteristic formula and Tate local duality for Galois cohomology.
\end{proof}

\subsection{Cohomology of rank 1 $\m$-modules}
In this section, we provide an explicit computation of $H^2$ of rank
1 $\m$-modules over the Robba ring in case $K=\Qp$ and $p>2$ as a
complement to Colmez's results on $H^0$ and $H^1$. Although we don't
need this for the main theorems, it is useful for some purposes (see
\cite[Lemma 2.3.11]{BC07}). In this section, all $\m$-modules are
over the Robba ring and $K=\Qp$. Moreover, to be consistent with
Colmez's set up, we fix $L$ a finite extension of $\Qp$ as the
coefficient field. This means we consider $\m$-modules over
$\r_{\Qp}\otimes_{\Qp}L$, where $\varphi$ and $\Gamma$ act on $L$
trivially, $\gamma(T)=(1+T)^{\chi(\gamma)}-1$ and
$\varphi(T)=(1+T)^p-1$. Following Colmez's notation, we use $\r_L$
to denote $\r_{\Qp}\otimes_{\Qp}L$ in this section only. Note that
this is different from our usual definition of $\r_L$.

If $\delta$ is a continuous character from $\Q^{\times}$ to $L^{\times}$, we can associate a rank $1$ $\m$-module \emph{$\r(\delta)$} to $\delta$.
Namely, there is a basis $v$ of $\r(\delta)$ such that
\begin{center}
$\varphi(xv)=\delta(p)\varphi(x)v$ and $\gamma(xv)=\delta(\chi(\gamma))\gamma(x)v$
\end{center}
for any $x\in \r_L$. Here $\chi$ is the cyclotomic character. It is
obvious that such $v$ is unique up to a nonzero scalar of $L$. In
the sequel, for $a\in \r_L$, we use $a$ to denote the element $av$
of $\r(\delta)$. Conversely, Colmez (\cite[Proposition 4.2, Remark
4.3]{C05}) proved that if $D$ is a $\m$-module of rank $1$, then
there is a unique character $\delta$ such that $D$ is isomorphic to
$\r(\delta)$.

For simplicity, let $H^i(\delta)$ denote $H^i(\r(\delta))$.
In the following, the character $x$ is the identity character
induced by the inclusion of $\Q$ into $L$ and $|x|$ is the character
mapping $x$ to $p^{-v_p(x)}$. We use $\omega$ to denote $x|x|$; then $\r(\omega)=\D^\dagger_\rig(L(1))$, as described in the
introduction.

In \cite{C05}, Colmez computed $H^0$ and $H^1$ for all the $\m$-modules of rank $1$ when $p>2$. More precisely, he proved the following result (\cite[Proposition 3.1, Theorem 3.9]{C05}).
\begin{proposition}If $p>2$, then the following are true.
\begin{enumerate}
\item[(1)]For $i\in \mathbb{N}$, $H^0(x^{-i})=L\cdot t^i$. If $\delta\neq x^{-i}$ for $i\in\mathbb{N}$, then $H^0(\delta)=0$.
\item[(2)]For $i\in \mathbb{N}$, $H^1(x^{-i})$ is a $2$-dimensional $L$-vector
space generated by $(0, \overline{t^i})$ and $(\overline{t^i}, 0)$, and $H^1(\omega x^i)$ is also $2$-dimensional. If $\delta\neq x^{-i}$ or $\omega x^i$ for $i\in \mathbb{N}$, then
$H^1(\delta)$ is $1$-dimensional.
\end{enumerate}
\end{proposition}
In fact, we can follow Colmez's method to compute $H^2$ easily. For $f=\sum_{k\in \mathbb{Z}}a_kT^k\in\r_L$, we define
the residue of the differential form $\omega=fdT$ by the formula $\res(\omega)=a_{-1}$. Then $\omega$ is closed if and only if
$\res(\omega)=0$.
Recall that $\partial=(1+T)\frac{d}{dT}$, then $\ker(\partial)=L$ and $df=\partial f\frac{dT}{1+T}$.
We define $\Res(f)=\res(f\frac{dT}{1+T})$, then $f$ is in the image of $\partial$ if and only if $\Res(f)=0$.

Recall that we have the following formulas:
\begin{center}
$\partial\circ\varphi=p\varphi\circ\partial$ and $\partial\circ\gamma=\chi(\gamma)\gamma\circ\partial$.
\end{center}
If $\Res(f)=0$, then there exists a $g\in\r_L$ such that $\partial(g)=f$.
Therefore we have $\partial(\frac{1}{p}\varphi(g))=\varphi(f)$ and $\partial(\frac{1}{\chi(\gamma)}\gamma(g))=\gamma(f)$.
Hence $\Res(\varphi(f))=\Res(\gamma(f))=0$. In general, if $\Res(f)=a\in L$, then $\Res(f-a(\frac{1+T}{T}))=0$. So
$\Res(\gamma(f))=a\Res(\gamma(\frac{1+T}{T}))$ and $\Res(\varphi(f))=a\Res(\varphi(\frac{1+T}{T}))$.
Note that $\log\frac{\gamma(T)}{T}$ and $\log\frac{\varphi(T)}{T^p}$ are defined in $\r_L$. Hence
\begin{center}
$0=\res(d\log\frac{\gamma(T)}{T})=\res(\frac{d\gamma(T)}{\gamma(T)}-\frac{dT}{T})=
\res(\frac{\chi(\gamma)(1+T)^{\chi(\gamma)-1}dT}{\gamma(T)})-1=\Res(\chi(\gamma)\gamma(\frac{1+T}{T}))-1$,\\
$0=\res(d\log\frac{\varphi(T)}{T^p})=\res(\frac{d\varphi(T)}{\varphi(T)}-\frac{pdT}{T})=
\res(\frac{p(1+T)^{p-1}dT}{\varphi(T)})-p=\Res(p\varphi(\frac{1+T}{T}))-p$,
\end{center}
therefore $\Res(\gamma(\frac{1+T}{T}))=1/\chi(\gamma)$ and $\Res(\varphi(\frac{1+T}{T}))=1$. So we get
\begin{center}
$\Res(\gamma(f))=1/\chi(\gamma)\Res(f)$ and $\Res(\varphi(f))=\Res(f)$.
\end{center}

For any $x\in \r_L$, by the formulas
$\partial\circ\varphi=p\varphi\circ\partial$ and $\partial\circ\gamma=\chi(\gamma)\gamma\circ\partial$, we have
\begin{center}
$\partial((x^{-1}\delta)(p)\varphi(x))=\delta(p)\varphi(\partial(x))$\\
$\partial((x^{-1}\delta)(\chi(\gamma))\gamma(x))=\delta(\chi(\gamma))\gamma(\partial(x))$.
\end{center}
So $\partial$ induces an $L$-linear morphism, which commutes with $\varphi$ and $\Gamma$,
from $\r(x^{-1}\delta)$ to $\r(\delta)$ by mapping $x$ to $\partial x$.
Then $\partial$ induces an $L$-linear morphism from $H^i(x^{-1}\delta)$ to $H^i(\delta)$.
\begin{proposition}
In case $p>2$, if $v_p(\delta(p))<0$, then $H^2(\delta)=0$.
\end{proposition}
\begin{proof}
For any $\bar{f}\in H^2(\delta)$, from \cite[Corollary 1.3]{C05}, there
is a $b\in\r_L$ such that $c=f-(\delta(p)\varphi-1)b$ lies in
$(\calE^\dagger_L)^{\psi=0}$. Since $\bar{c}=\bar{f}$ in
$H^2(\delta)$, we can just assume $f$ is in
$(\calE^\dagger_L)^{\psi=0}$. Then there exists an $a\in\calE^\dagger_L$
such that $f=(\delta(\chi(\gamma))\gamma-1)a$ by Lemma 2.2. Therefore $\bar{f}=0$ in $H^2(\delta)$.
\end{proof}

\begin{proposition}If $p>2$, then the following are true.
\begin{enumerate}
\item [(1)]
If $\delta\neq x$, then $\partial: H^2(\delta
x^{-1})\rightarrow H^2(\delta)$ is injective.
If $\delta\neq \omega$, then $\partial: H^2(\delta
x^{-1})\rightarrow H^2(\delta)$ is surjective. Therefore $\partial: H^2(\delta
x^{-1})\rightarrow H^2(\delta)$ is an isomorphism if $\delta\neq \omega, x$.
\item [(2)]$H^2(\omega)$ is a $1$-dimensional $L$-vector space, generated by $\overline{1/T}$.
\item [(3)]$H^2(x^k)=0$ for any $k\in\mathbb{Z}$. Combining with $(1)$, we conclude that $\partial$ is always injective.
\end{enumerate}
\end{proposition}
\begin{proof}
For $(1)$, first suppose $\delta\neq x$. If $\partial(\bar{f})=0$ for some $\bar{f}\in H^2(x^{-1}\delta)$, this means that
there exist $a$, $b\in\r_L$ such that $\partial(f)=(\delta(\chi(\gamma))\gamma-1)a-(\delta(p)\varphi-1)b$.
Now since $\Res(\partial(f))=0$, we have $\Res((\delta(\chi(\gamma))\gamma-1)a)=\Res((\delta(p)\varphi-1)b)$.
Therefore
\begin{center}
$(\delta(\chi(\gamma))\chi(\gamma)^{-1}-1)\Res(a)=(\delta(p)-1)\Res(b).$
\end{center}
If $\delta(p)-1=0$, then $v_p(\delta x^{-1}(p))<0$. Therefore $H^2(\delta x^{-1})=0$ by Proposition 2.11.
If $\delta(p)-1$ is not zero,
let $c=(\delta(p)-1)^{-1}\Res(a)\frac{1+T}{T}$, $a'=a-(\delta(p)\varphi-1)c$ and
$b'=b-(\delta(\chi(\gamma))\gamma-1)c$. Then we have $\Res(a')=\Res(b')=0$
and $\partial(f)=(\delta(\chi(\gamma))\gamma-1)a'-(\delta(p)\varphi-1)b'$.

So we can assume that $\Res(a)=\Res(b)=0$. Now suppose
$\partial(\tilde{a})=a$ and $\partial(\tilde{b})=b$.
Let $\tilde{f}=f-((\delta(\chi(\gamma))\chi(\gamma)^{-1}\gamma-1)\tilde{a}-(\delta(p)p^{-1}\varphi-1)\tilde{b})$,
then $\partial(\tilde{f})=0$. This implies $\tilde{f}\in L$. Since $\delta\neq x$, we have either
$\delta(\chi(\gamma))\chi(\gamma)^{-1}-1\neq 0$ or $\delta(p)p^{-1}-1\neq 0$.
If $\delta(\chi(\gamma))\chi(\gamma)^{-1}-1\neq 0$,
let $\tilde{a}'=\tilde{a}+(\delta(\chi(\gamma))\chi(\gamma)^{-1}-1)^{-1}\tilde{f}$, then
$f=(\delta(\chi(\gamma))\chi(\gamma)^{-1}\gamma-1)\tilde{a}'-(\delta(p)p^{-1}\varphi-1)\tilde{b}$. So $\bar{f}=0$ in
$H^2(\delta x^{-1})$. If $\delta(p)p^{-1}-1$ is not
zero, let $\tilde{b}'=\tilde{b}-(\delta(p)p^{-1}-1)^{-1}\tilde{f}$, then
$f=(\delta(\chi(\gamma))\chi(\gamma)^{-1}\gamma-1)\tilde{a}-(\delta(p)p^{-1}\varphi-1)\tilde{b}'$.
So $\bar{f}$ is also zero.

If $\delta\neq \omega$,
then either $\delta(\chi(\gamma))\chi(\gamma)^{-1}-1$ or $\delta(p)-1$ is not zero.
Hence for any $\bar{f}\in H^2(\delta)$, we can choose $a$, $b$ such
that $\Res(f-(\delta(\chi(\gamma))\gamma-1)a-(\delta(p)\varphi-1)b)=0$.
Then there exists an $f'$ such that
$\partial(f')=f-(\delta(\chi(\gamma))\gamma-1)a-(\delta(p)\varphi-1)b$. So $\partial(\bar{f'})=\bar{f}$.
This proves the surjectivity.

For $(2)$, we have that both
$\omega(\chi(\gamma))\chi(\gamma)^{-1}-1$ and $\omega(p)-1$ are zero. So
we can define a map $\Res: H^2(\omega)\rightarrow L$ by
$\Res(\bar{f})=\Res(f)$. We claim that it is an isomorphism. If
$\Res(\bar{f})=0$, then $\bar{f}$ is in the image of $\partial:
H^2(|x|)\rightarrow H^2(\omega)$. But $H^2(|x|)=0$ by Proposition 2.11, so $\bar{f}=0$. Therefore $\Res$ is injective.
Note that $\Res(\overline{1/T})=1$, so $\Res$ is also surjective.

For $(3)$, if $k<0$, then $H^2(x^{k})=0$ by proposition 2.11. If $k\in\mathbb{N}$,
then $\partial:H^2(x^{k-1})\rightarrow H^2(x^k),... ,\partial:H^2(1)\rightarrow H^2(x)$ and
$\partial:H^2(x^{-1})\rightarrow H^2(1)$ are all surjective by $(1)$.
So $H^2(x^k)=0$ since $H^2(x^{-1})=0$.
\end{proof}
\begin{corollary}Suppose $p$ is not equal to $2$. If $\delta=\omega x^k$ for $k\in \mathbb{N}$, then $H^2(\delta)$ is a
$1$-dimensional $L$-vector space generated by $\overline{\partial^k(1/T)}$. Otherwise, $H^2(\delta)=0$.
\end{corollary}
\begin{proof}
This is an easy consequence of Propositions 2.11 and 2.12. In fact, for any $\delta$, we can find a $k_0\in\mathbb{N}$
such that $v_p(\delta x^{-k_0}(p))<0$. Then $H^2(\delta x^{-k_0})=0$. If $\omega$ does not appear in the sequence
$\delta x^{-k_0}$,..., $\delta x^{-1}$, $\delta$, then $H^2(\delta)=0$ by Proposition 2.12(1).
If $\omega$ appears, then $\delta=\omega x^k$ for some $k\in \mathbb{Z}$. If $k<0$, then $H^2(\delta)=0$ since
$v_p(\omega x^k(p))<0$. For $k\in\mathbb{N}$, by Proposition 2.12(2), $H^2(\omega )$ is generated by $\overline{1/T}$.
Repeatedly applying $(1)$ of Proposition 2.12, we get that $H^2(\omega x^k)$ is generated by $\overline{\partial^k(1/T)}$.
\end{proof}
\begin{corollary}
If $p>2$, then the Euler-Poincar\'e characteristic formula holds for all rank 1 $\m$-modules .
\end{corollary}

\section{Generalized $\m$-modules}
In the rest of this paper, all $\m$-modules are over the Robba ring. For simplicity, we only consider the usual Robba ring without an additional coefficient field. However, it is easy to see that the same argument works to prove the results in the general case.

\subsection{Generalized $\m$-modules}
In this section we will investigate generalized $\m$-modules. Define
a \emph{generalized $\m$-module over $\r_K$} as a finitely presented
$\r_K$-module $D$ with commuting $\varphi, \Gamma$-actions such that
$\varphi^*D \to D$ is an isomorphism. Since $\r_K$ is a Bezout
domain (\cite[Proposition 4.6]{Cr98}), it is a coherent ring (i.e.
the kernel of any map between finitely presented modules is again
finitely presented), so the generalized $\m$-modules form an abelian
category. Define a \emph{torsion $(\phi, \Gamma)$-module} as a
generalized $(\varphi,\Gamma)$-module which is $\r_K$-torsion. We
say a torsion $\m$-module $S$ is a \emph{pure $t^k$-torsion
$\m$-module} if it is a free $\r_K/t^k$-module. For a generalized
$\m$-module $D$, its torsion part $S$ is a torsion $\m$-module and
$D/S$ is a $\m$-module. We define the rank of $D$ as the rank
of $D/S$.
\begin{proposition}
If $K=\Qp$, then a torsion $\m$-module $S$ is a successive extensions of pure $t$-torsion $\m$-modules.
\end{proposition}
\begin{proof}
From \cite[Proposition 4.12(5)]{LB02}, we can find a set of elements
$\{e_1$,...,$e_d\}$ of $S$ and principal ideals
$(r_1),(r_2),...,(r_d)$ of $\r_K$ such that $S=\oplus_{i=1}^d \r
e_i$, Ann$(e_i)=(r_i)$, and
$(r_1)\subset(r_2)\subset...\subset(r_d)$. Furthermore, these ideals
$(r_1),(r_2),...,(r_d)$ are unique. Therefore they are
$\Gamma$-invariant. Since $\r_K$ is a free $\r_K$-module via
$\varphi$, we have Ann$(1\otimes e_i)=$Ann$(e_i)=(r_i)$ in
$\varphi^*S$ for every $i$. Hence Ann$(\varphi(e_i))=(r_i)$ because
$\varphi^*S \to S$ is an isomorphism. This implies
$(\varphi(r_i))\subset$Ann$(\varphi(e_i))=(r_i)$.

We claim that if a principal ideal $I$ of $\r_{\Q}$ is stable under
$\varphi$ and $\Gamma$, then it is $(t^k)$ for some $k$. In fact,
from the proof of \cite[Lemma I.3.2]{LB04b}, since $I$ is stable
under $\varphi$ and $\Gamma$ it is generated by
$\prod_{n=1}^\infty(\varphi^{n-1}(q)/p)^{j_n}$ for a decreasing
sequence $\{j_n\}_n$. Therefore these $j_n$'s are eventually
constant, let $k\in\mathbb{N}$ denote this constant. This implies
$I=(t^k)$. So we conclude that $(r_i)=(t^{k_i})$ for every $i$, and
$\{k_i\}_i$ is decreasing. Then we can construct a filtration
$0=t^{k_1}S\subset t^{k_1-1}S\subset...\subset S$ of $S$ such that
each quotient is a pure $t$-torsion $\m$-module, hence the result.

\end{proof}

In case $K=\Qp$, for any pure $t^k$-torsion $\m$-module $S$, let $d=\rank_{\r/t^k}S$, and choose a basis
$\{e_1$,...,$e_d\}$ of $S$.
Let $A$ be the matrix of $\varphi$ in this basis. Since $\varphi^{*}S\cong S$, there is another matrix $B$ such that
$AB=BA=I_d$. Furthermore, since $\Gamma$ is topologically finite generated, we can choose an $r_0$ large enough
such that $A$, $B$ and the elements of $\Gamma$ have all entries lie in
$\btrig{,r_0}{,\Q}/(t^k)$
For $r\geq r_0$, set $S_r$ be the $\btrig{,r}{,\Q}/(t^k)$-submodule of $S$ spanned by $\{e_1$,...,$e_d\}$.
Then $\Gamma$ acts on $S_r$ and $\varphi:S_r\ra S_{pr}$ induces an isomorphism
\[
1\otimes\varphi:\btrig{,pr}{,\Q}/(t^k)\otimes_{\btrig{,r}{,\Q}/(t^k)}S_r\cong S_{pr}.
\]
Here we view $\btrig{,pr}{,\Q}/(t^k)$ as a $\btrig{,r}{,\Q}/(t^k)$ algebra via $\varphi$.

\begin{lemma}
For $r\geq p-1$, we have the following.
\begin{enumerate}
\item[(1)] The natural maps $\btrig{,r}{,\Q}/(t^{k})\rightarrow
\btrig{,r}{,\Q}/(\varphi^{n}(q^{k}))$ for $n\geq n(r)$ induce an isomorphism
\begin{center}
$\btrig{,r}{,\Q}/(t^{k})\cong \prod^{\infty}_{n\geq n(r)}
\btrig{,r}{,\Q}/(\varphi^{n}(q^{k}))$.
\end{center}

\item[(2)] If $n\geq n(r)$, the localization $\pi\mapsto
\eps^{(n)}e^{t/p^{n}}-1$ induces a $\Gamma$-equivariant
isomorphism from $\btrig{,r}{,\Q}/(\varphi^{n}(q^{k}))$ to
$\L[t]/(t^{k})$.

\item[(3)] For $r'\geq r$, $\btrig{,r}{,\Q}/(\varphi^{n}(q^{k}))\rightarrow
\btrig{,r'}{,\Q}/(\varphi^{n}(q^{k}))$ is the identity map via
the isomorphism of (2).

\item[(4)] The morphism $\varphi: \btrig{,r}{,\Q}/(t^k)\rightarrow \btrig{,pr}{,\Q}(t^k)$ can
be described via the isomorphism of (1) as follows:
$\varphi((x_{n})_{n\geq n(r)})=((y_{n})_{n\geq n(r)+1})$ where
$y_{n+1}=x_{n}$ for $n\geq n(r)$.
\end{enumerate}
\end{lemma}
\begin{proof}
 See \cite[Lemma 3.15]{C05}.
\end{proof}
Using $(2)$ of Lemma 3.2, for $n\geq n(r)$, we set $S^n=S_r\otimes_{\btrig{,r}{}/(t^k)} \L[t]/(t^k)$.
 Then $S^n$ is a free $\L/(t^k)$-module of rank $d$ with $\Gamma$-action.
The injective map $\varphi:S^n\ra S^{n+1}$  induces an isomorphism
\[
1\otimes\varphi:\mathbb{Q}_p(\varepsilon^{(n+1)})[t]/(t^k)\otimes_{\L[t]/(t^k)}S^n\cong S^{n+1}.
\]
It allows us to regard $S^n$ as a submodule of $S^{n+1}$.
\begin{theorem} With notations as above, the following are true.
\begin{enumerate}
\item[(1)]The natural maps $S_r\rightarrow S^n$
for $n\geq n(r)$ induce $S_r\cong \prod^{\infty}_{n\geq n(r)}S^n$ as
$(\mathbb{Q}_p(\varepsilon^{(n(r))})[t]/(t^k))[\Gamma]$-modules.\

\item[(2)]For $r'\geq r$, under the isomorphism of (1), the natural map
$S_r\rightarrow S_{r'}$ is $((x_{n})_{n\geq n(r)})\mapsto((x_{n})_{n\geq n(r')})$.\

\item[(3)]Under the isomorphism of (1), $\varphi:S_r\rightarrow
S_{pr}$ is $(x_{n})_{n\geq n(r)}\mapsto((y_{n})_{n\geq n(r)+1})$,
where $y_{n+1}=x_{n}$ for $n\geq n(r)$.\
\end{enumerate}
\end{theorem}
\begin{proof}
For $(1)$, we have
\begin{eqnarray*}
S_r&=&S_r\otimes_{\btrig{,r}{,\Q}}\btrig{,r}{,\Q}/(t^k)  \\
&=&S_r\otimes_{\btrig{,r}{,\Q}} \prod^{\infty}_{n\geq n(r)}\btrig{,r}{,\Q}/((\varphi^{n}(q))^k) \qquad(\text{by (1) of Lemma 3.2})\\
&=&\prod^{\infty}_{n\geq n(r)}S_r\otimes_{\btrig{,r}{,\Q}}\btrig{,r}{,\Q}/((\varphi^{n}(q))^k)\\
&=&\prod^{\infty}_{n\geq n(r)}S_r\otimes_{\L[t]}\L[t]/(t^k) \qquad(\text{by (2) of Lemma 3.2})\\
&=&\prod^{\infty}_{n\geq n(r)}S^n.\\
\end{eqnarray*}
Then (2) and (3) follow from (3) and (4) of Lemma 3.2 respectively.
\end{proof}

The natural
examples of torsion $\m$-modules are \emph{quotient $\m$-modules} which are of the forms $D/E$,
here $E\subset D$ are two $\m$-modules of the same rank.
For sufficiently large $r$, we have
$E_{r}\subset D_{r}$. For $n\geq n(r)$, localizing at
$\eps^{(n)}-1$, we get $D^{+,n}_{\dif}(E)\subset
D^{+,n}_{\dif}(D)$. It is natural to view the quotient
$D^{+,n}_{\dif}(D)/ D^{+,n}_{\dif}(E)$ as the localization at
$\eps^{(n)}-1$ of the quotient $\m$-module $D/E$. The connecting map $\varphi_{n}$ of $D$ and $E$ induces a connecting map
\[
\varphi_n: D^{+,n}_{\dif}(D)/ D^{+,n}_{\dif}(E)\ra D^{+,n}_{\dif}(D)/ D^{+,n}_{\dif}(E)
\] of $D/E$ such that
\[
\varphi_{n}\otimes 1: D^{+,n}_{\dif}(D)/ D^{+,n}_{\dif}(E)\otimes_{\L[t]}\mathbb{Q}_p(\varepsilon^{(n+1)})[t]/(t^k)\ra
D^{+,n}_{\dif}(D)/ D^{+,n}_{\dif}(E)
\]
is an isomorphism. Note that if we apply the proof of Theorem 3.2 to $D/E$
by replacing $S_r$ by $D_r/E_r$ and $S^n$ by $D^{+,n}_{\dif}(D)/ D^{+,n}_{\dif}(E)$, then we get the following formulas
which might be useful for other purposes.
\begin{proposition}Suppose $K=\Qp$. With notation as above, the following are true.
\begin{enumerate}
\item[(1)]The localization maps $D_{r}/E_{r}\rightarrow D^{+,n}_{\dif}(D)/ D^{+,n}_{\dif}(E)$
induce an isomorphism
\begin{center}
$D_{r}/E_{r}\cong \prod^{\infty}_{n\geq n(r)}D^{+,n}_{\dif}(D)/ D^{+,n}_{\dif}(E)$
\end{center}
as $(\mathbb{Q}_p(\varepsilon^{(n(r))})[t]/(t^k))[\Gamma]$ modules.
\item[(2)]For $r'\geq r$, under the isomorphism of (1), the natural map
$D_{r}/E_{r}\rightarrow D_{r'}/E_{r'}$ is $((x_{n})_{n\geq n(r)})\mapsto((x_{n})_{n\geq n(r')})$.

\item[(3)]Via the isomorphism of (1), $\varphi: D_{r}/E_{r}\rightarrow
D_{pr}/E_{pr}$ is $((x_{n})_{n\geq n(r)})\mapsto((y_{n})_{n\geq n(r)+1})$,
where $y_{n+1}=x_{n}$ for $n\geq n(r)$.
\end{enumerate}
\end{proposition}

\subsection{Cohomology of Generalized $\m$-modules}
We also use Herr's complex to define the cohomology of generalized
$\m$-modules. If $L$ is a finite extension of $K$, and $D$ is a
generalized $\m$-module over $\r_{L}$, then we define the
\emph{induced $\m$-module of $D$} from $L$ to $K$ in the same way as
for $\m$-modules as in the beginning of section 2.2 and also denote
it by $\Ind_{L}^{K}D$.
\begin{theorem}(Shapiro's Lemma for generalized $\m$-modules)
Suppose $D$ is a $\m$-module over $\r_{L}$. Then there are
isomorphisms
\[
H^i(D)\cong H^i(\Ind_{L}^KD)   \qquad (i=0,1,2)
\]
which are functorial in $D$ and compatible with cup products.
\end{theorem}
\begin{proof}The proof is the same as the proof of Theorem 2.2.
\end{proof}

Suppose $\eta: \Zp^{*}\rightarrow \OO_{L}$ is a character of finite
order with conductor $p^{N(\eta)}$ ($N(\eta)=0$ if $\eta=1$;
otherwise it is the smallest integer $n$ such that $\eta$ is trivial
on $1+p^{n}\Zp$). We define the Gauss sum $G(\eta)$ associated to
$\eta$ by $G(\eta)=1$ if $\eta=1$, otherwise
\[
 G(\eta)=\sum_{x\in (\ZZ/p^{N(\eta)}\ZZ)^{*}} \eta(x)\mu_{p^{N(\eta)}}^{x} \in
L_{N(\eta)}^{\times}.
\]
\begin{lemma}
Let $k$ be a positive integer.
\begin{enumerate}
 \item[(1)] If $\eta: \Zp^{*}\rightarrow \OO_{L}$ is of finite order and $0\leq
i\leq k-1$, then $g(G(\eta)t^{i})=(\eta^{-1}\chi^{i})(g)\cdot(G(\eta)t^{i})$ for
every $g\in \Gamma$.
 \item[(2)]For every $n\in \mathbb{N}$, we have
$\L[t]/t^{k}=\oplus_{\eta, N(\eta)\leq n}\oplus_{0\leq i\leq
k-1}\Qp\cdot G(\eta)t^{i}$.
 \end{enumerate}
\end{lemma}
\begin{proof}
See \cite[Prop 3.13] {C05}.
\end{proof}
\begin{theorem}

Suppose $S$ is a torsion $\m$-module. Then we have the following.
\begin{enumerate}
\item[(1)]$\dim_{\Qp}H^{0}(S)=\dim_{\Qp}H^{1}(S)<\infty$;
\item[(2)]$\varphi-1$ is surjective on $S$, and therefore $H^{2}(S)=0$.
\end{enumerate}
\end{theorem}

\begin{proof}By Theorem 3.5, we first reduce the theorem to the case $K=\Qp$. Suppose we are given a short exact sequence $0\rightarrow S'\rightarrow S\rightarrow S''\rightarrow 0$ of torsion $\m$-modules and the theorem holds for $S'$ and $S''$. Then $(2)$ also holds for $S$ by Five Lemma. From the long exact sequence of
cohomology we get
\[
 0\rightarrow H^0(S')\rightarrow H^{0}(S)\rightarrow H^{0}(S'')\rightarrow H^{1}(S')\rightarrow H^1(S)
 \rightarrow H^1(S'')\rightarrow 0.
\]
Then we see that $\dim_{\Qp}H^{0}(S)=\dim_{\Qp}H^{1}(S)<\infty$ since $\dim_{\Qp}H^{0}(S_i)=\dim_{\Qp}H^{1}(S_i)<\infty$ for $i=1$, $2$.

So conditions $(1)$ and $(2)$ are preserved by extensions.
By Proposition 3.1 we only need to treat the case in which $S$ is a pure $t^k$-torsion $\m$-module.
We claim that the map $\varphi-1:S_r\rightarrow S_{pr}$ is surjective for any $r\geq p-1$. This will prove $(2)$,
since $S$ is the union of $S_r$'s. In fact, for any $((y_{n})_{n\geq n(r)+1})\in S_{pr}$,
if we let $x_n=-\sum_{i=n(r)}^{n}y_i$ for $n\geq n(r)$, where we put $y_{n(r)}=0$, then we have
$(\varphi-1)((x_n)_{n\geq n(r)})=((y_{n})_{n\geq n(r)+1})$ by $(2)$ and $(3)$ of Theorem 3.3.

For $(1)$, we set $S_r'=S_r^{\triangle_K}$ and $(S^n)'=(S^n)^{\triangle_K}$.
Then we have that $S_r'\cong \prod^{\infty}_{n\geq n(r)}(S^n)'$. By Theorem 3.3(2), if $a=((a_{n})_{n\geq n(r)})\in S_r'$,
then $a=0$ if and only if $a_{n}=0$ for almost all $n$. For any $a\in
H^{0}(S)$, suppose $a$ is represented by $((a_{n})_{n\geq n(r)})\in
S_r'$; then $(\varphi-1)(a)=0$ implies $a_{n}$
becomes constant for $n$ large enough. Therefore we have
\[
H^{0}(S)=\limind_{n\rightarrow \infty }((S_n)')^{\Gamma/\triangle_K} =\limind_{n\rightarrow \infty }(S_n)^\Gamma.
\]

Suppose $(a, b)\in Z^{1}(S)$. By $(2)$ which we have proved, there exists a $c\in S$ such that $(\varphi-1)c=b$.
Then $(a, b)$ is homogeneous to $(a-(\gamma_K-1)c, 0)$, so we can assume that $b=0$.
Suppose $a$ is represented by $((a_{n})_{n\geq n(r)})\in S_r'$ for
some $r$. Then $(\varphi-1)a=0$ implies $a_{n}$ becomes constant
for $n$ is large enough, say for $n\geq n_{0}$. Also $(a,0)$ is a coboundary if and only if $a_{n}\in
(\gamma_K-1)(S^n)'$ for some $n\geq n_{0}$. Then we have
\[
H^{1}(S)=\limind_{n\rightarrow \infty }(S^n)'/(\gamma_K-1).
\]
Since $(S^n)'$ is a finite dimensional $\Qp$-vector space, we have that
$\dim_{L}(S^n)^{\Gamma}=\dim_{L}(S^n)'/(\gamma_K-1)$. Since $(S^n)^{\Gamma}\ra
(S^{n+1})^{\Gamma}$ is injective by Lemma 3.3(2), in order to prove $(1)$, we need only to
verify two things: $(a)$ $(S^n)'/(\gamma-1)\rightarrow
(S^{n+1})'/(\gamma-1)$ is injective, and $(b)$
$\dim_{L}(S^n)^{\Gamma}$ has an upper bound independent of $n$. \

From Lemma 3.6(2), $\L[t]/(t^k)$ is a direct summand of
$\mathbb{Q}_p(\varepsilon^{(n+1)})[t]/(t^k)$ as $\Gamma$-modules.
Hence $S^n$ is a direct summand of $S^{n+1}$ as $\Gamma$-modules,
then $(S^n)'$ is also a direct summand of $(S^{n+1})'$ as
$\Gamma$-modules and this proves $(a)$.

For $s\in\mathbb{N}$, using Lemma 3.6(2) for
$S^{n+s}=\mathbb{Q}_p(\varepsilon^{(n+s)})[t]/(t^k)\otimes_{\L[t]/(t^k)}S^n$,
we see
\[
 \dim_{\Qp}(S^{n+s})^{\Gamma}\leq \sum_{N(\eta)\leq n+s}\sum_{0\leq
i\leq k-1}\dim_{\Qp}(S^n(\eta^{-1}\chi^{i}))^{\Gamma}).
\]
But the right hand side is no more than $\dim_{\Qp}S^n$ because
these characters $\eta^{-1}\chi^{i}$ are distinct. Hence
$\dim_{\Qp}(S^{n+s})^{\Gamma}\leq \dim_{\Qp}S^n$ for any $s$, and this proves $(b)$.
\end{proof}
\begin{corollary}
For any torsion $\m$-module $S$, we have $\chi(S)=0$.
\end{corollary}
\section{Main Theorems}
\subsection{Euler-Poincar\'e characteristic formula} The main goal of
this section is to prove the Euler-Poincar\'e characteristic formula.
\begin{lemma}
For any $\m$-module $D$ and $0\leq i\leq2$, $\dim_{\Qp}H^i(D)$ is finite if and only if $\dim_{\Qp}H^i(D(x))$ is finite. Furthermore, if all of $\dim_{\Qp}H^i(D)$ are finite, then $\chi(D)=\chi(D(x))$.
\end{lemma}
\begin{proof}We identify $D(x)$ with $tD$, then apply Theorem 3.7 and Corollary 3.8 to $D/D(x)$.
\end{proof}

\begin{lemma}We can find a $\m$-module $E$ of rank $d$ such that
\begin{enumerate}
\item[(1)]$E$ is pure and $\mu(E)=1/d$;
\item[(2)]$E$ is a successive extensions of $\r(x^i)$'s, where $i$ is either $0$ or $1$.
\end{enumerate}
\end{lemma}
\begin{proof}We proceed by induction on $d$. For $d=1$, take $E=\r(x)$. Now
suppose $d>1$ and the lemma is true for $d-1$. Choose such an example $E_{0}$. By Lemma 4.1, we have
$\chi(\r(x))=\chi(\r)$. Since $E_{0}$ is a successive extensions of $\r(x^i)$'s where $i$ is either $0$ or $1$, we have
\[
 \dim_{\Qp}H^{1}(E_{0})\geq-\chi(E_{0})=(\rank E_{0})(-\chi(\r))=(d-1)[K:\Qp]\geq1,
\]
where the last equality follows from Corollary 2.9. Therefore we can
find a nontrivial extension $E$ of $\r_K$ by $E_{0}$. Then
$\mu(E)=1/d$. We claim that $E$ is pure. In fact, suppose $P$ is a
submodule of $E$ such that $\mu(P)<1/d$. Since $\rank P\leq d$, we
get $\deg P\leq 0$, and hence $\mu(P)\leq 0$. Therefore, $P\cap
E_{0}=0$, since $E_{0}$ is pure of positive slope. Therefore the
composite map $P\ra E\ra \r_K$ is injective, we get that
$\mu(P)\geq0$ with equality if and only if it is an isomorphism
\cite[Corollary 1.4.10]{Ke06}. But this forces the extension to be
trivial, which is a contradiction. Obviously $E$ also satisfies
$(2)$, so we finish the induction step.
\end{proof}

\begin{theorem}(Euler-Poincar\'e characteristic formula) For any generalized $\m$-module $D$, we have
\begin{enumerate}
\item[(1)]$\dim_{\Qp}H^{i}(D)<\infty$ for $i=0,1,2$
\item[(2)] $\chi(D)=-[K:\Q]\rank D$.
\end{enumerate}
\end{theorem}
\begin{proof}First by Theorem 3.6, we reduce to the case $D$ is a $\m$-module. Then by Theorem 2.2, we can further reduce to the case where $K=\Qp$. We first show that $\dim_{\Qp}H^{0}(D)\leq d=\rank D$ for any $D$.
For $r$ large enough, $D_{r}$ is defined and we have
$D_{r}\hookrightarrow D_{\dif}^{+,n}(D)[1/t]$ for $n\geq n(r)$. We
claim that $\dim_{\Qp}(D_{\dif}^{+,n}(D)[1/t])^{\Gamma}\leq d$.
Otherwise we can find $e_{1}, e_{2},..., e_{d+1}\in
(D_{\dif}^{+,n}(D)[1/t])^{\Gamma}$ that are linearly
independent over $\Qp$. But $D_{\dif}^{+,n}(D)[1/t]$ is a
$d$-dimensional vector space over $\L((t))$. So $e_{1},
e_{2},..., e_{d+1}$ are linearly dependent over $\L((t))$. Then
there is a minimal $k$ such that $k$ of these vectors are linearly
dependent over $\L((t))$. Assume $e_{1}, e_{2},..., e_{k}$ are
$k$ such vectors and $\sum_{i=1}^{k}a_{i}e_{i}=0$. Obviously
$a_{1}\neq 0$ since $k$ is minimal, so
$e_{1}+\sum_{i=2}^{k}(a_{i}/a_{1})e_{i}=0$. Using $\gamma$, we get
$e_{1}+\sum_{i=2}^{k}\gamma(a_{i}/a_{1})e_{i}=0$. By minimality of
$k$, we must have $\gamma(a_{i}/a_{1})=a_{i}/a_{1}$. But $(\L((t)))^{\Gamma}=\Qp$, so $e_{1}, e_{2},..., e_{k}$
are linearly dependent over $\Qp$. That is a contradiction. So we get
$\dim_{\Qp}(D_{r})^{\Gamma}\leq d$ for any $r$; therefore
$\dim_{\Qp}H^{0}(D)\leq\dim_{\Qp}D^{\Gamma}\leq d$.

We will prove Theorem 4.3 by induction on the $\rank$ of $D$. Assume for some $d\geq 1$ the theorem holds for all $\m$-modules
which have $\rank$ less than $d$. Now suppose $\rank D=d$. Note that both of $(1)$ and $(2)$ are preserved under
taking extensions. Thus by the slope filtration theorem we can further assume that $D$ is pure.
Suppose $\mu(D)=c/d$. Let $E$ be as in Lemma 4.2. Then $(\otimes ^{s}_{i=1} E)(x^k)$ is pure of slope $k+s/d$. In particular,
we can find a pure $\m$-module $F$ which is a successive extensions
of $\r (x^i)$'s and $\mu(F)=-c/d$. Consider the \'etale $\m$-module $D\otimes F$.
By Corollary 2.9 we get
\begin{center}
 $\chi(D\otimes F)=-\rank(D\otimes F)=-\rank F\rank D$.
\end{center}

On the other hand, by the construction of $F$, $D\otimes F$ is
a successive extensions of $D(x^i)$'s. So in particular there exists a
$j\in\ZZ$ such that $D(x^j)$ is a saturated submodule of $D\otimes F$.
Let $G$ be the quotient, so we have the long exact sequence of
cohomology:
\[
 \cdots\rightarrow H^{0}(G)\rightarrow H^{1}(D(x^j))\rightarrow H^{1}(D\otimes
F)\rightarrow\ H^1(G)\rightarrow H^2(D(x^j))\rightarrow H^2(D\otimes F)\cdots.
\]
Since $D\otimes F$ is \'etale, $\dim_{\Qp}H^{1}(D\otimes F)$ is finite
by Theorem 2.6. Hence  $\dim_{\Qp}H^{1}(D(x^j))$ is finite; then $\dim_{\Qp}H^{1}(D(x^i))$ is finite for any $i\in\ZZ$.
If $\dim_{\Qp}H^2(D)=\infty$, by Lemma 4.1 we have that $\dim_{\Qp}H^2(D(x^i))=\infty$ for any $i$.
This implies $\dim_{\Qp}H^2(D\otimes F)=\infty$ from the above sequence. But this is a contradiction since $D\otimes F$ is \'etale. Therefore $\dim_{\Qp}H^2(D(x^i))$ is finite for any $i$. By Lemma 4.1, we have that $\chi(D)=\chi(D(x^i))$. By the additivity of $\chi$, we get
\begin{center}
 $\chi(D\otimes F)=(\rank F)\chi(D)$;
\end{center}
hence
\[
 (\rank F)\chi(D)=-\rank F\rank D,
\]
\begin{center}
$\chi(D)=-\rank D$.
\end{center}
The induction step is finished.
\end{proof}

\subsection{Tate local duality theorem} The main topic of this
section is to prove the Tate local duality theorem: the cup product
\[
H^i(D) \times H^{2-i}(D^\vee(\omega))\to H^2(\omega)\cong \Qp
\]
is a perfect pairing for any $\m$-module $D$ and $0 \leq i\leq 2$.
\begin{lemma}
Suppose $0\rightarrow D'\rightarrow D\rightarrow D''\rightarrow0$ is an exact
sequence of  $\m$-modules. If Tate local duality holds for any two
of them, it also holds for the third one.
\end{lemma}
\begin{proof}First note that the pairing $H^i(D) \times H^{2-i}(D^\vee(\omega))\to H^2(\omega)\cong \Qp$
is perfect if and only if the induced map
$H^{2-i}(D^\vee(\omega)) \to H^i(D)^\vee$ is an isomorphism.
From the long exact sequence of cohomology, we get the following
commutative diagram.

\[
\xymatrix{
... \ar[r] & H^{2-i}(D'^\vee(\omega))\ar[d] \ar[r]
& H^{2-i}(D^\vee (\omega)) \ar[d] \ar[r]
& H^{2-i}(D'^\vee(\omega)) \ar[d] \ar[r] & ... \\
... \ar[r] & H^i(D'')^\vee \ar[r] & H^i(D)^\vee \ar[r] &
H^i(D')^\vee \ar[r] & ... }
\]
Then the lemma follows from the Five Lemma.
 \end{proof}
\begin{lemma}Tate local duality is true for $\r(|x|)$.
\end{lemma}
\begin{proof}
By the Euler-Poincar\'e formula, we get $\dim_{\Qp}H^1(x|x|^{-1})\geq-\chi(\r(x^{-1}|x|))=[K:\Qp]$.
Hence there exists a nonsplit short exact sequence of $\m$-modules
\[
0\ra\r(x)\longrightarrow D\longrightarrow\r(|x|)\longrightarrow 0.
\]
Then $\deg(D)=\deg(\r(x))+\deg(\r(|x|))=0$ and furthermore we see $D$ is forced to be \'etale. In fact, suppose $P$ is a
submodule of $D$ such that $\mu(P)<0$, then $P$ is necessary of $\rank$ $1$ and hence $\mu(P)\leq-1$.
Then $P\cap \r(x)=0$, hence $P$ maps injectively to $\r(|x|)$. So we have $\mu(P)\geq-1$. Therefore we conclude that
$\mu(P)=-1$; but this forces $P$ to map isomorphically to $\r(|x|)$, which is a contradiction. If $a\in H^0(x)$, then $\varphi(a)=a/p$.
It implies $\varphi(at)=at$, yielding $at$ is a constant, therefore $a=0$. If $a\in H^0(|x|)$,
then $\gamma(a)=a$ for any $\gamma\in\Gamma$, so $a$ is a constant. But $\varphi(a)=pa$, hence
$a=0$. So $ H^0(D) = 0$ by the long exact sequence of cohomology. Take the dual exact sequence
\[
0\longrightarrow\r(x)\longrightarrow D^\vee(\omega)\longrightarrow\r(|x|)\longrightarrow 0.
\]
By usual Tate duality, $H^0(D)$ is dual to $H^2(D^\vee(\omega))$, so $H^2(D^\vee(\omega)) = 0$.
Hence $H^2(|x|)=0$, so $\dim_{\Qp}H^1(|x|)=[K:\Qp]$ by the Euler-Poincar\'e formula.
The cup pairing gives a morphism of long exact sequences:
\[
\xymatrix{
...\ar[r]&0\ar[d]\ar[r] & H^{1}(x)\ar[d] \ar[r]
& H^{1}(D^\vee (\omega)) \ar[d] \ar[r]
& H^{1}(|x|) \ar[d] \ar[r] &H^2(x)\ar[d]\ar[r]& ... \\
... \ar[r] & H^2(x)^\vee \ar[r] & H^1(|x|)^\vee \ar[r] &
H^1(D)^\vee \ar[r] & H^1(x)^\vee\ar[r]&0\ar[r]&... }
\]
in which $H^1(D^\vee(\omega))\ra H^1(D)^\vee$ is an isomorphism by usual Tate duality.
Then diagram chasing shows that $H^1(x)\ra H^1(|x|)^\vee$ is injective, so
$H^1(x)$ has $\Qp$-dimension $\leq [K:\Qp]$. Then by the Euler-Poincar\'e formula, $\dim_{\Qp}H^1(x)=[K:\Qp]$ and $H^2(x) = 0$.
Therefore $H^1(x)\ra H^1(|x|)^\vee$ is an isomorphism.
\end{proof}
\begin{remark}Note that we can use $\r(x^{-1})$ instead $\r(|x|)$ in the proof of Theorem 4.7 (see below). In case $K=\Qp$ and $p>2$, we can verify the Tate duality for $\r(x^{-1})$ by explicit calculations. Recall that $\Res: H^2(\omega)\rightarrow \Qp$ is an isomorphism.
For $i=0$, $H^0(x^{-1})=\Qp\cdot t$, $H^2(\omega x)=\Qp\cdot
\overline{(1+T)/T^2}$, the cup product of $t$ and $\overline{(1+T)/T^2}$ is
$\overline{t(1+T)/T^2}$, and $\Res(t(1+T)/T^2)=1$. For $i=1$, $H^1(x^{-1})$ has a basis$\{(\bar{t}, 0), (0, \bar{t})\}$.
From \cite[Proposition 3.8]{C05}, $H^1(\omega)$ has a basis $\{(\bar{a},
\overline{1/T+1/2}), (\overline{1/T},\bar{b})\}$, where $a\in T\r^+$
and $b\in (\calE^\dagger)^{\psi=0}$. Furthermore,
$\partial:H^1(\omega)\rightarrow H^1(\omega x)$ is an isomorphism,
therefore $\{(\overline{\partial a}, \overline{-(1+T)/T^2}),
(\overline{-(1+T)/T^2},\overline{\partial b})\}$ is a basis of
$H^1(\omega x)$. A short compuatation shows that under the given
basis, the matrix of cup product is
$\left(
\begin{smallmatrix}
 1&* \\
 0&-1
\end{smallmatrix}
\right)$. For $i=2$, there is nothing to say since
$H^2(-1)=H^0(\omega x)=0$.
\end{remark}
\begin{theorem} The Tate local duality is true for all
$\m$-modules.
\end{theorem}
\begin{proof}
By Lemma 4.4 and the slope filtration theorem, we need only to prove the
theorem for pure $\m$-modules. Suppose $D$ is a pure $\m$-module of
rank $d$. By passing to $D^{\vee}(\omega)$, we can further assume
$\mu(D)=s/d\geq 0$. We proceed by induction on $s=\deg(D)$.  \

If $s=0$, $D$ is \'etale, so the theorem follows from Corollary 2.6. Now
suppose $s>0$ and the theorem is true for any pure $\m$-module $D$
which satisfies $0\leq\deg(D)< s$. By the Euler-Poincar\'e formula, we have
$\dim_{\Qp}H^{1}(D(|x|^{-1}))\geq d \geq 1$. Hence we can find a nontrivial
extension $E$ of $\r(|x|)$ by $D$. Then $\deg(E)=s-1$ and
$\mu(E)=(s-1)/(d+1)<\mu(D)$. Suppose the slope filtration of $E$ is $0=E_{0}\subset
E_{1}\subset...\subset E_{l}=E$. Then $\mu(E_{1})\leq
\mu(E)<\mu(D)$. Note that $\deg(E_1) = \deg(E_1 \cap D) + \deg(E_1/(E_1 \cap D))$. Since
 $D$ is pure of positive slope, $\deg(E_1 \cap D) > 0$ unless $E_1 \cap D
 = 0$. Since $E_1/(E_1 \cap D)$ is the image of
 $E_1 \to \mathcal{R}(|x|)$, $\deg(E_1/(E_1 \cap D)) \geq -1$.
 Consequently, $\deg(E_1) \geq 0$ unless $E_1 \cap D = 0$ and $E_1/(E_1\cap D) \cong \mathcal{R}(|x|)$, but these imply that the extension splits, which it does not by construction. So we have $\mu(E_{j}/E_{j-1})\geq
0$ for each $j$. Note that $\sum_{j=1}^{l}\deg(E_{j-1}/E_{j})=\deg
E=s-1$. Thus for each $j$, we have $\deg(E_{j-1}/E_{j})<s$. Hence
$E_{j-1}/E_{j}$ satisfies the theorem by induction. Therefore the
theorem is true for $E$ by Lemma 4.4. By Lemma 4.5 the theorem holds for $\r(|x|)$. Therefore the same is true for $D$ by Lemma 4.5 again.
This finishes the induction step.
\end{proof}
\begin{remark} Our approach to Tate local duality is similar to the way we
established the Euler-Poincar\'e formula. However, in the case of Tate local
duality, Euler-Poincar\'e formula has provided the existence of
nontrivial extensions, so we don't need the reduction steps on
torsion $\m$-modules that were used in the proof of Theorem 4.3.
\end{remark}

\subsection*{Acknowledgments}
Many thanks are due to my advisor, Kiran Kedlaya, for many useful
discussions and crucial suggestions, and for spending much time
reviewing early drafts of this paper. In this last regard, thanks
are due as well to Jay Pottharst. The author would also like to
thank Ga\"{e}tan Chenevier for helpful conversations at the Clay
Eigensemester of 2006 which inspired the author to start this
project. Thanks Laurent Berger for useful correspondence. Thanks
finally to Chris Davis for helping with the English of this paper.
Any remaining mistakes are my own.

\end{document}